\documentclass[12pt]{amsart}
\usepackage{amsmath, amssymb, amsthm, latexsym}
\usepackage{tikz-cd}
\usepackage{amssymb,amscd,amsmath}
\usepackage{latexsym}
\usepackage[center]{caption}
\usepackage{tikz}
\usepackage{tikz}
\usepackage{mathrsfs}
\newcounter{braid}
\newcounter{strands}

\DeclareMathAlphabet{\bsf}{OT1}{cmss}{bx}{n}

\pgfkeyssetvalue{/tikz/braid height}{1cm}
\pgfkeyssetvalue{/tikz/braid width}{1cm}
\pgfkeyssetvalue{/tikz/braid start}{(0,0)}
\pgfkeyssetvalue{/tikz/braid colour}{black}
\pgfkeys{/tikz/strands/.code={\setcounter{strands}{#1}}}

\makeatletter
\def\cross{%
  \@ifnextchar^{\message{Got sup}\cross@sup}{\cross@sub}}

\def\cross@sup^#1_#2{\render@cross{#2}{#1}}

\def\cross@sub_#1{\@ifnextchar^{\cross@@sub{#1}}{\render@cross{#1}{1}}}

\def\cross@@sub#1^#2{\render@cross{#1}{#2}}

\def\render@cross#1#2{
  \def\strand{#1}
  \def\crossing{#2}
  \pgfmathsetmacro{\cross@y}{-\value{braid}*\braid@h}
  \pgfmathtruncatemacro{\nextstrand}{#1+1}
  \foreach \thread in {1,...,\value{strands}}
  {
    \pgfmathsetmacro{\strand@x}{\thread * \braid@w}
    \ifnum\thread=\strand
    \pgfmathsetmacro{\over@x}{\strand * \braid@w + .5*(1 - \crossing) * \braid@w}
    \pgfmathsetmacro{\under@x}{\strand * \braid@w + .5*(1 + \crossing) * \braid@w}
    \draw[braid] \pgfkeysvalueof{/tikz/braid start} +(\under@x pt,\cross@y pt) to[out=-90,in=90] +(\over@x pt,\cross@y pt -\braid@h);
    \draw[braid] \pgfkeysvalueof{/tikz/braid start} +(\over@x pt,\cross@y pt) to[out=-90,in=90] +(\under@x pt,\cross@y pt -\braid@h);
    \else
    \ifnum\thread=\nextstrand
    \else
     \draw[braid] \pgfkeysvalueof{/tikz/braid start} ++(\strand@x pt,\cross@y pt) -- ++(0,-\braid@h);
    \fi
   \fi
  }
  \stepcounter{braid}
}

\tikzset{braid/.style={double=\pgfkeysvalueof{/tikz/braid colour},double distance=1pt,line width=2pt,white}}

\newcommand{\braid}[2][]{%
  \begingroup
  \pgfkeys{/tikz/strands=2}
  \tikzset{#1}
  \pgfkeysgetvalue{/tikz/braid width}{\braid@w}
  \pgfkeysgetvalue{/tikz/braid height}{\braid@h}
  \setcounter{braid}{0}
  \let\sigma=\cross
  #2
  \endgroup
}
\makeatother

\input xypic
\newtheorem{theorem}{Theorem}
\newtheorem{proposition}[theorem]{Proposition}

\newtheorem{comment}[theorem]{Comment}
\newtheorem{lemma}[theorem]{Lemma}

\newtheorem{definition}[theorem]{Definition}


\def\Z{\mathbb{Z}}

\def\N{\mathbb{N}}

\def\md{\mathcal{D}}

\def\qed{\hfill$\square$\medskip}

\def\Zpk{\mathbb{Z}/p^{k}}
\def\Zpk1{\mathbb{Z}/p^{k-1}}

\newcommand{\rref}[1]{(\ref{#1})}

\newcommand{\beg}[2]{\begin{equation}\label{#1}#2\end{equation}}
\def\r{\rightarrow}

\def\sl2{\widetilde{SL_{2}(\Z)}}

\def\md
\def\rank{\operatorname{rank}}

\title[]{Actads}
\author{Sophie Kriz}

\begin{document}
\maketitle

\begin{abstract}
In this paper, I introduce a new generalization of the concept of an operad, further generalizing the concept of an {\em opetope} 
introduced by Baez and Dolan, who used this for the definition of their version of non-strict $n$-categories. 
Opetopes arise from iterating a certain construction on operads called the $+$-construction, starting with monoids. The
first step gives rise to plain operads, i.e. operads without symmetries. The permutation axiom in a symmetric operad, however, is an additional
structure resulting from permutation of variables, independent of the structure of a monoid.
Even though we can apply the $+$-construction to symmetric operads, there is the possibility of introducing a completely different kind of permutations
on the higher levels by again permuting variables without regard to the structure on the previous levels. 
Defining and investigating these structures is the main purpose of this paper.
The structures obtained in this way is what I call {\em $n$-actads}.
In $n$-actads with $n>1$, the permutations on the different levels
give rise to a certain special kind of $n$-fold category. I also explore the concept
of iterated algebras over an $n$-actad (generalizing an algebra and module over an operad), and various types of iterated units.
I give some examples of algebras over $2$-actads, and show how they can be used to construct certain
new interesting homotopy types of operads. I also discuss a connection between actads and ordinal notation. 

\end{abstract}

\newpage

\vspace*{10mm}
\tableofcontents
\newpage

\section{Introduction}

\vspace{5mm}

{\em Operads} are a central concept of modern algebraic topology. The term was coined by J. P. May \cite{MayIterated}. Operads are used as
an approach to infinite loop space theory. 
They are a more streamlined alternative to the PROPs of Adams and MacLane, Lawvere theories, and the concepts of
Boardman and Vogt.
An operad $\mathscr{C}$ consists of sets (or spaces) $\mathscr{C}(n)$ where $\mathscr{C}(n)$
is the set of all operations in $n$ variables in a type of algebraic structure, which it describes. 
The basic operations of an operad correspond to substituting
the results of operations (in different variables) to the input variables of another operation. One also typically includes a unit (identity)
operation and commutativity in the broader sense, which means a symmetric group action by permuting variables.
All this structure is included in May's original definition of an {\em operad} (\cite{MayIterated}, Definition 1.1, which also included the assumption that
$\mathscr{C}(0)=*$).
The concept without commutativity was defined by May (\cite{MayIterated}, Definition 3.12) as {\em non-$\Sigma$ operads}.
Some authors also speak of {\em plain operads}. In contrast, operads where permuting variables is allowed are
sometimes called {\em symmetric operads}.
Another variant is an {\em $S$-sorted operad} (or {\em multicategory}), 
which is defined in the same way as an operad, but with an ``object set" $S$.

In the original context of infinite loop space theory, May \cite{MayIterated} used the {\em little $n$-cube operad} which naturally acts on an
$n$-fold loop space. By passing to the limit, an infinite loop space is shown to be an {\em $E_{\infty}$-space}, which means an algebra over an operad $\mathscr{C}$ where each $\mathscr{C}(n)$ is a contractible space
with free $\Sigma_n$-action. This is one way of capturing the notion of a commutative monoid ``up to all possible
homotopies."

In a completely different context, operads also later appeared in algebra, where instead of operations on a set or a space, we axiomatize multilinear operations on a vector space. In this fashion, for example, commutative and associative algebras, as well as Lie and Poisson algebras can be
axiomatized. Ginzburg and Kapranov (\cite{GinzKapronav}) discovered a striking phenomenon of {\em Koszul duality} of operads and their algebras, analogous to previously known concept of Koszul duality of algebras and modules \cite{Priddy}.

Passing from vector spaces to their chain complexes, we can also talk about chain-level $E_{\infty}$-operads and their algebras. Hinich and
Schechtman \cite{HS} noticed that the cochain complex of a topological space has the structure of an $E_{\infty}$-algebra (which they called
a {\em May algebra}). Mandell \cite{p-adic} proved that for well-behaved spaces (for example simply connected with finitely generated homotopy groups), their $p$-completed homotopy type can be recovered from their $E_{\infty}$ cochain complex with coefficients in $\bar{\mathbb{F}}_p$.

In some sense, the ideas of the topological and algebraic contexts are combined in {\em multiplicative infinite loop space theory} (cf. \cite{MayMultiplicativeA, EM}), which generalizes
certain structures of multilinear algebra to stable homotopy theory, i.e. to the context of spectra (see \cite{spectrum}, Part III).

\vspace{5mm}

Baez and Dolan (see \cite{BaezDolan}) introduced the {\em $+$-construction} (unrelated to Quillen's $+$-construction), which is a procedure used to pass from monoids to plain operads. 
Essentially, for an
$S$-sorted operad $\mathscr{O}$, there is an operad $\mathscr{O}^+$ sorted over the set of operations of $\mathscr{O}$, whose algebras are
operads over $\mathscr{O}$. Applying this procedure to monoids defines plain operads, and iterating defines {\em opetopes}, which Baez and Dolan
\cite{BaezDolan} used to axiomatize non-strict $n$-categories.
Variants and generalizations of these concepts were introduced by C. Hermida, M. Makkai, and J. Power \cite{HermMakPow1, HermMakPow2, HermMakPow3}, E. Cheng \cite{ChengOpet, ChengWeak, ChengWeakOpes}, and T. Palm \cite{PalmBook, PalmPaper} . 
Additional references include M. Zawadowski \cite{ZFace, ZFacePositive, ZLax} and M. Fiore \cite{MFioreSaville}.

The $+$-construction was further generalized to Cartesian monads by Leinster \cite{Leinster}, and interpreted 
combinatorially by Kock, Joyal, Batanin, and
Mascari \cite{KockJoyalBatMascari}, using the calculus of polynomial functors of Gambino and Hyland \cite{GambinoHyland}. 
An even more conceptual interpretation in this direction was given by Szawiel and Zawadowski \cite{SzawielZawadowskiTheory, SzawielZawadowskiWeb,SzawielZawadowski}. For another approach using syntactic methods, see Curien, Thanh, and Mimram \cite{CurienThanhMimram}. Computer implementations of polytopes are given in \cite{KockJoyalBatMascari, Finster}.

However, opetopes do not explain how symmetric
operads arise from monoids. While the $+$-construction can be applied to symmetric operads, there is an additional possibility of introducing
permutations on each higher level by permuting variables without regard to the structures on the previous levels. 
The
interplay of these permutations is the concept I explore in this paper. I call the resulting structures {\em actads}.
Note that the possibility of permuting variables randomly does not fit within the picture of any reasonable concept of a higher category: it is a truly new element. Perhaps for that reason, actads have so far defied the conceptual approaches described above, and the only rigorous definition I was able to make starts from scratch.

As it turns out, the higher permutations must be handled quite carefully. We must specify a delicate order of variables on the previous levels
which must be preserved for the higher permutations to work. Permutations on the different levels cannot just be mixed into an ordinary category,
but must form a special kind of $n$-fold category. Because of this, a particularly canonical model of the structure without permutations
(which I call {\em plain actads}) must be introduced. Plain $n$-actads form a category equivalent to a version of $n$-opetopes (\cite{BaezDolan}),
a combinatorial model of which was given in \cite{KockJoyalBatMascari}. However, we need an ordered model which will be introduced from scratch.

In this paper, I index the actads so that monoids are $0$-actads and operads
are $1$-actads. The operations of an actad are indexed by an {\em $n$-base}.
The $0$-base is just $1$ point, the $1$-base is the set  of natural numbers. Elements of the $2$-base are,
roughly speaking, based trees. Drawing, in the place of each node of a tree, a triangle whose vertex is the node,
and whose base contains the successor nodes, one can visualize a $3$-base as a tree drawn in this fashion, where each triangle
can be, recursively (but only finitely many times), subdivided into a tree drawn in a similar fashion.
Remembering all these subdivisions is a key part of the structure of an element of the $3$-base.
For $n>3$, $n$-bases are more difficult to visualize.

To explain how the induction procedure works in more detail, I find it easier to first only
discuss associativity.
I call the concept only encoding associativity a {\em plain actad}.
Plain actads are not categories, but merely sets. I define a plain $n$-actad using an indexing set $B_n$ called the {\em plain $n$-base}.
For example, the elements of $B_2$ are {\em planar trees},
which means that the successors to each node are linearly ordered.
In addition to the $B_n$'s, there are maps $F_n$, which map $B_n$ into the set of non-empty finite
sequences in $B_{n-1}$, and $G_n$, which simply maps $B_n$ into $B_{n-1}$. 
Now, the $F_n$'s can be thought of as assigning to an element of $B_n$ the ``set of its entries."
For example, $F_2$ can be visualized as giving the degrees of internal nodes of a planar tree.
In addition, $G_n$ can be thought of as assigning to an element of $B_n$ its ``total arity", which is an element of $B_{n-1}$. Finally, we can define a composition operation by taking an element $x\in B_n$ and replacing one of its ``entries" (i.e. one of the terms $z$ of $F_n(x)$) with some $y\in B_n$ such that $G_n(y)=z$. For example, we begin with $B_0 = \{*\}$ and $B_1 = \mathbb{N}$ (where we identify a natural number $n$ with a planar
tree with $1$ triangle and $n$ prongs). Composition of $n,m \in \mathbb{N}=B_1$ is then defined to always give $n+m-1$
and can be visualized in the diagram \rref{PlainB1Composition}, below. Note that we always have
$F_1 (n) = (\underbrace{*, \dots, *}_n)$ and $G_1(n) = *$ for every $n$, so all $n,m\in B_1$
are composable in all ways.

Now, a plain $n$-actad is a system $\mathcal{C}(z)$ indexed by $z\in B_n$. The structure of a plain $n$-actad is
defined by composition operations in $\mathcal{C}$ ``fibered over $B_n$", i.e. where the indices are elements involved in  a composition in $B_n$. Finally, we can derive the associativity axioms from associativity properties of the composition in $B_n$.

To advance the induction, I define $B_{n+1}$ as the free plain $n$-actad on $B_n$. In other words, we no longer ``execute" the composition operations on $B_n$, but simply keep track of which ones were made, subject to the associativity axiom. In this fashion, we can inductively define the $B_n$'s, and thus we can define plain $n$-actads. 
I treat plain $n$-actads below in Section \ref{s1}.

Just as in the case of operads, capturing a commutativity axiom complicates the structure substantially. In the case of an operad $\mathscr{C}$, a commutativity property is expressed by introducing a $\Sigma_n$-action on $\mathscr{C}(n)$. Then this turns $B_1$ into a groupoid $\mathscr{B}_1$,
making $\Sigma_n$ the automorphism group of $n\in\N$. From a philosophical point of view, $n$ really stands for the word
\beg{wordan}{a^n=a\cdots a}
in the free monoid on one element $a$, and the permutations interchange (independently of the monoid structure) the $n$ copies of $a$. 
In this paper, I develop the corresponding additional structure $\mathscr{B}_2$, and eventually $\mathscr{B}_n$.
Recall that the elements of $B_2$ are planar trees. $\mathscr{B}_2$ should incorporate the ``$1$-permutations" induced from $\mathscr{B}_1$ by
\rref{wordan}, which are isomorphisms of trees. Notice that a $1$-permutation can change the tree, so it may no longer be an automorphism, although it
is always an isomorphism (so $1$-permutations form a groupoid).
For example, see the picture of a $1$-permutation in $\mathscr{B}_2$ in \rref{1permutationof2trees}.
But now it seems that any natural generalization should introduce permutations of nodes of the tree of the same arity (call them $2$-permutations), regardless of the structure of the
tree (just as in \rref{wordan}, the permutations are completely independent to the product operation of the
free monoid). For example, see \rref{2permutationof2trees}.

We need to answer the question of how the $1$-permutations and $2$-permutations interact. Putting them both
together into a single category seems to fail. In other words, a random permutation of nodes of the same arity in a tree does not seem to induce,
in any meaningful way, a permutation on leaves. Nor does there really seem to be a natural structure of a $2$-category: Both $1$-permutations
and $2$-permutations act on objects. Thus, we are led to a bicategory structure (see \cite{Bicategory}), where the $\{1,2\}$-morphisms are permutations
of nodes of equal arity, on which we have prescribed the same permutation of successors. This, fortunately, is a rather special type of bicategory. In addition to the $\{1\}$-morphisms and $\{2\}$-morphisms (we may in the future omit the braces to simplify notation) forming groupoids,
every $1$-morphism and $2$-morphism with the same source complete to a unique $\{1,2\}$-morphism. In this paper, I call this type of bicategory
{\em cube-like}, and generalize it to all $n$. Then, $\mathscr{B}_n$ is, in particular, a cube-like $n$-fold category. Along with the appropriate axioms expressing compatibility with composition, this is my implementation of the commutativity axiom. Cube-like $n$-fold
categories fibered over $\mathscr{B}_n$ with composition satisfying the properties of $\mathscr{B}_n$'s composition are what I call {\em $n$-actads}.
I treat this in Section \ref{s2}.

The question of units is another issue. There is, of course, a natural concept of unit that can be considered a part of associativity, which generates two
unit axioms. However, there is more to the story. Even in operads (which, recall, are $1$-actads), we encounter another concept, namely that of a {\em based operad}, which also contains a
manifestation of the unit of a monoid (which is a $0$-actad). More precisely, a based operad $\mathscr{C}=(\mathscr{C}(n))$ is indexed by $n\in\N_0$, and we are given a base point $*\in \mathscr{C}(0)$. Based operads play a crucial role in infinite loop space theory (for example, in the Approximation Theorem \cite{MayIterated}).

Similarly, for $n$-actads, beside the ``ordinary" unit notion there is a notion which includes carry-overs from $k$-actad units for $k<n$. I call this
``recursive" unit axiom {\em R-unitality}. It is interesting to note that in the case of $n$-actads, the commutativity
axiom becomes substantially richer in the $R$-unital case. This is because for $n\geq 3$, in the non-unital case, an element of an $n$-actad can be only non-trivially composed with elements of lower arity, with which it cannot be permuted by an $\{n\}$-morphism. In the $R$-unital context, this is not necessarily the case.
I treat the $R$-unital case in Section \ref{sunital} below.

The original application of operads comes from their algebras. In addition, the notion of algebras prompted the notion of a module over an operad and its algebras, which was invented by May, and is treated, for example, in \cite{CH}.
We would like to generalize these notions to $n$-actads. Indeed, there are directly analogous notions for $n$-actads, and I treat these concepts in Section
\ref{salg}. An important example is the {\em commutative} (in the narrower sense) $(n+1)$-actad $\mathscr{B}_{n+1}$ whose algebras can be thought of as strictly commutative $n$-actads. By ``freeing up" the $(n+1)$-morphisms, we are able to obtain the {\em associative} $(n+1)$-actad $\mathscr{A}_{n+1}$, whose algebras are precisely $n$-actads. 

For $n\geq 2$, the immediate thought is that this idea should be ``iterated." However, implementing that idea is not obvious. Before we can iterate, we must for example answer the question of what happens to $n$ in the iterated algebras. The answer to these questions are given in Section \ref{siter}, and are summarized as follows:

For an $n$-actad $\mathscr{C}$, it turns out that there exists a notion of an {\em iterated $\mathscr{C}$-$I$-algebra} for every subset $I\subseteq\{0,\dots,n\}$, which consists of $|I|$ cube-like $(n-1)$-fold categories. An axiomatization of these notions is made using multisorted algebras
(a definition of which can be found in \cite{multi}).
The concept essentially allows one to label the entries of a tree with the elements of $I$ (for
example, see \rref{IteratedExamplePicture}). In the case of $I = \{0\}$ and $I=\{0,n\}$, the concept gives the notions of algebras and modules, respectively.
I treat this in Section \ref{siter}.

What are some interesting examples and applications of $n$-actads?
So far, we have been working in the ground category of sets when discussing actads and their algebras. One can of course also work with spaces or simplicial sets without any difficulties (as is done with operads in \cite{MayIterated}). This is where one may expect the most interesting topological
examples to reside. At present, finding such examples is restricted by the fact that, for $n\geq 2$, the structures are unfamiliar, but we do present
some examples of the new structure here. One interesting example is the {\em $E_{\infty}$-$n$-actad} $E\mathscr{A}_n$, which is the \v{C}ech resolution of the associative $n$-actad
$\mathscr{A}_n$. Note that since we have a canonical map of $n$-actads
\beg{+}{\mathscr{A}_n\r E\mathscr{A}_n,}
$E\mathscr{A}_n$-algebras are, in particular, $(n-1)$-actads.

For $n=2$, I give some concrete examples related to this case. I describe the free $E\mathscr{A}_2$-algebra $A_{E\mathscr{A}_2}(X_2)$ on the
groupoid $X_2$ fibered over $\mathscr{B}_1$
where $X_2(2)$ is the groupoid given by $\Sigma_2$ acting on itself freely and $X_2(n)$ is empty over $n\neq 2$.
I also give a sufficient condition for the classifying space of a groupoid to be algebras over the operad
$A_{E\mathscr{A}_2}(X_2)$. In some sense, this is an analogue of the result of May \cite{MaySpectra} that the classifying space of a permutative category is an
algebra over the Barratt-Eccles $E_{\infty}$-operad, which is $E\mathscr{A}_1$.

In Section \ref{snew}, I actually construct $E\mathscr{A}_2$-algebras which are new operads (in the sense that they, as far as I know, are not
equivalent to operads which have been previously studied). For a spectrum $E$ (see \cite{spectrum} for an introduction the concept
of spectra in algebraic topology) with a map $E\r H\Z$ (the Eilenberg-Mac Lane spectrum), we can obtain a map of $E_{\infty}$-spaces $E_0\r\Z$. I show that this structure (also when $E_0$ is replaced with any $E_{\infty}$-space),
in particular, leads to a structure of an $E\mathscr{A}_2$-algebra, and show that in certain cases, the resulting example is non-trivial even as an operad, (via \rref{+}), in the sense that an $E_{\infty}$-operad does not map into it.

In fact, algebras over a based version $\bar{\Xi}_{E_0}$ of this operad can be characterized as follows: 
For a map of spectra $E\rightarrow F\rightarrow H\mathbb{Z}$, let $\psi:F_0\rightarrow\mathbb{Z}$ be the corresponding map of
$E_{\infty}$-spaces. Then $\bar{\Xi}_{E_0}$-algebras are canonically equivalent to spaces of the form $\psi^{-1}(-1)$.
Note that $\psi^{-1}(-1)$ is homotopy equivalent to $\psi^{-1}(0)$, which is an infinite loop space, but the homotopy
equivalence is non-canonical when the map of spectra $E\rightarrow H\mathbb{Z}$ does not split.

While this is only one example of a homotopical structure obtained from the $2$-morphisms of actads, its non-triviality
can be thought of as higher-categorical in nature (since it comes from a lack of canonicity).
Thus, our example may provide a clue to the nature of the kinds of applications the new structure of an $n$-actad
may have.

There is yet an entirely different aspect of the story of plain actads. Kock, Joyal, Batanin, Mascari (\cite{KockJoyalBatMascari})
discuss a machine interpretation and also a ``fixed point" of the $+$-construction. It is possible to pursue this further.
Our particular model of plain actads suggests a close connection with the Veblen hierarchy of ordinal notation. From this point of view,
if is possible to iterate the opetope construction transfinitely and we show that the level of complexity of the first fixed point of that construction is the Feferman-Sch{\"u}tte ordinal $\Gamma$.
More specifically, recalling the fact that the elements
of $B_2$ can be identified with planar trees, one is reminded of the Veblen hierarchy of ordinals, where,
letting $\varphi_\beta$ denote the Veblen functions (see \cite{Veblen}), $\varphi_1(\alpha)=\omega^\alpha$ and
$\varphi_{\beta}(\alpha)$ is the $\alpha$'th fixed point of $\varphi_{\beta'}$ for $\beta'<\beta$. (This is shifted by $1$ from the usual convention, which is more natural for my point of view. The shift disappears for $\beta\geq\omega$.)
The first fixed point of this hierarchy (the first ordinal with $\varphi_{\Gamma}(0)=\Gamma$) is the
Feferman-Sch\"{u}tte ordinal $\Gamma$ \cite{Veblen}.
It is particularly well known that trees can be used for a notation of ordinals below $\varepsilon=\varphi_2(0)$ which suggests a connection with plain actads. 
In Section \ref{sord}, I indeed construct an onto function
\beg{*}{B_n\r\varphi_n(0).}
This prompts the idea of generalizing $B_n$ to the case when $n$ is an ordinal, which is carried out readily.
Thus, we can also use $B_{\alpha}$, for $\alpha<\Gamma$, as notation for countable ordinals below $\Gamma$.
The map I constructed is not a bijection (for example even for trees, we have $1+\omega=\omega$). This could be remedied by taking a proper subset of $B_n$,
but we do not undertake this task in the present paper. Since actads are algebraic structures, even the map \rref{*} adds to visualization of ordinals between $\varepsilon$ and $\Gamma$.

\vspace{5mm}

\noindent{\bf Acknowledgment:}
I am grateful to the referees for many helpful suggestions for improving the presentation of this paper.

\vspace{5mm}

\section{Plain Bases and Actads}\label{s1}

\vspace{5mm}

In this section, we will introduce plain actads. By the results of \cite{KockJoyalBatMascari}, (following the concept of Gambino and Hyland
\cite{GambinoHyland}), the category of plain $n$-actads is equivalent to $n$-opetopes without units, (a unital version of this statement also holds, see
Section \ref{sunital}). However, plain actads give a canonical ordered model of the category which we will need when introducing actads.
As we described in the introduction, in the process of defining plain $n$-actads, we must also define the plain $n$-base
and the maps $F_n$'s and the $G_n$'s.
We will define a set $B_n$, also called the {\em plain $n$-base} and plain $n$-actads inductively.
Plain $n$-actads will be universal algebras sorted on $B_n$. 

In general, (see \cite{multi}), for a set $S$, an {\em $S$-sorted universal algebra} is defined to be the collection of
the data of a set $X$ with a map $X\r S$
together with a set $O_n$ for each $n\geq 0$ of $n$-ary operations (specific to the type of algebra we are considering) together
with a map $O_n\r S^{n+1}$ (specifying which operations are allowed with given sources and target), that is required to staisfy certain prescribed identities,
(also specific to the kind of algebra considered) where in compositions, the output of an operation is plugged into an input over the same $s\in S$
(i.e. compositions are
$$O_n\times_{S^n}(O_{k_1}\times\dots\times O_{k_n})\r O_{k_1+\dots+ k_n}).$$
Homomorphisms of a given kind of $S$-sorted algebras are maps over $S$ (i.e. commuting with the given
maps to $S$) which are compatible with the operations. For example, a group acting on a set forms a $\{0,1\}$-sorted
universal algebra $\varphi:X\r \{0,1\}$ where the group is $\varphi^{-1}(0)$ and the set is $\varphi^{-1}(1)$.
$S$-sorted algebras are equivalent to finitary monads in the category $(Set/S)$ of $Sets$ over $S$.
(The proof follows from the methods of Chapter 3 of \cite{AdamekRosicky}.)
This means a monad $M$ such that for any $X\r S$, $M(X\r S)$ is the colimit in $Set/S$ of $M(K\r S)$ with $K\subseteq X$ finite.

For a map $f:S\r T$, a $T$-sorted algebra specifies an $S$-sorted algebra
by allowing those operations and identifies which are allowed after composing with $f$.
For a finitary monad $M$ in $Set/T$,
we can describe the monad $f^*M$ in $Set/S$ as $RML$ where
$L: Set/S \r Set/T$ is the forgetful functor and $R:Set/T \r Set/S$ is its right adjoint, given by $R(X)=X\times_T S$.
In particular, an unsorted universal algebra can be made $S$-sorted by pulling back along the map $S\r *$.
Thus $S$-sorted operads (or multicategories) are defined.
Of course, in multicategories, we often consider multifunctors, which are more general morphisms over a map of sets $S_1\r S_2$
preserving the operations in the obvious sense. 
In general, for a $S$-sorted universal algebra pulled back via a map $f:S\r T$, we have
a notion of a morphism over any map $S_1\r S_2$ over $T$.

Plain $n$-actads are universal algebras sorted in $B_n$, this means that a plain $n$-actad is a set $X$ over $B_n$
(i.e. a map $X\r B_n$) with certain operations (which, as it will turn out, are binary) applicable to elements that map to certain
elements of $B_n$, specific to the operation, and certain identities among iterated operations.
For $n=0$, the plain $0$-base is just a point, 
$$B_0=\{*\}.$$

\vspace{3mm}

To describe the operations in more detail, we must also say more about the maps. For $n\geq 1$, additional structure on $B_n$ is needed,
which we will define inductively.
Write
$$T_{B_{n-1}} = \{(a_1,\dots , a_k) | a_i \in B_{n-1}, \; k\geq 1\}.$$
We have maps
$$F_n:B_n\r T_{B_{n-1}},$$
(let $m_x$ be the length of the sequence $F_n(x)$ for $x\in B_n$, and write
$$F_n(x)=(f_1(x),\dots, f_{m_x}(x))),$$ and
$$G_n:B_n\r B_{n-1}.$$
Suppose $B_{n-1}$ and $(n-1)$-actads have already been defined.
We then have the following structure:
We have
an operation
$$x\circ_i^n y,$$
for $x,y\in B_n$, $1\leq i\leq m_x$, and $G_n(y)= f_i(x)$, such that 
$$m_{x\circ_i^n y}= m_x+m_y-1.$$
To illustrate what we are after, before making exact definitions, let us consider what these structures look like for $n=0,1,2$.
For $n=0$, only consists of a single $B_0$ point, so one can only compose an element with itself leading to picture that is not
very interesting:

\setlength{\unitlength}{0.8cm}
$$
\begin{picture}(12,2)

\put(7,1){\circle*{.2}}

\end{picture}$$

However, for $n=1$ and $n=2$, the pictures are much more revealing. I will use the following pictures to express $1$- and $2$-compositions.
For $n=1$, $B_1$ is the set of all natural numbers, and we see the standard composition picture for planar trees
(identifying a natural number $n$ with the planar tree with $n$ prongs). Here, we are replacing one of the ``prongs" of a planar tree at the end of one of the trees with the ``stem" of a different tree. The ``prongs" are visualized as ordered ``left to right."

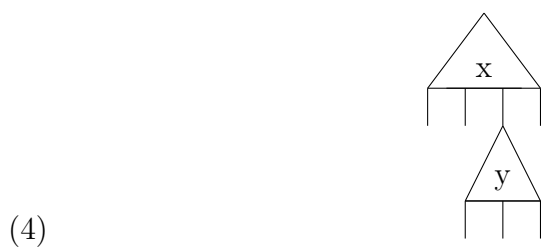
\beg{PlainB1Composition}{
\begin{tikzpicture}
\draw (0,1.5) -- (0,2);
\draw (0,2) -- (1.5,2);
\draw (0,2) -- (0.75,3);
\draw (0.75,3) -- (1.5,2);
\draw (0.5,1.5) -- (0.5,2);
\draw (1,1.5) -- (1,2);
\draw (1.5,1.5) -- (1.5,2);
\draw (1,1.5) -- (0.5,0.5);
\draw (0.5,0.5) -- (1.5,0.5);
\draw (1.5,0.5) -- (1,1.5);
\draw (0.5,0.5) -- (0.5,0);
\draw (1,0.5) -- (1,0);
\draw (1.5,0.5) -- (1.5,0);

\draw (0.25,2) -- node[above] {x} ++(1,0);
\draw (0.5,0.5) -- node[above] {y} ++(1,0);
\end{tikzpicture}}

It is important to note that, for operads, any two trees can be composed. This is because for every tree $x$, we have
$G_1(x)=*$ and $F_1(x)=(*,\dots,*)$. So, for every $1\leq i\leq m_x$, $f_i(x)=G_1(x)=*$. So every two trees, every possible $f_i$ will agree with $G_1$.

For $n=2$, though, the composition is quite different. In the picture, I express $2$-composition by putting a tree with the same arity inside one of the ``triangles."
Letting
$$
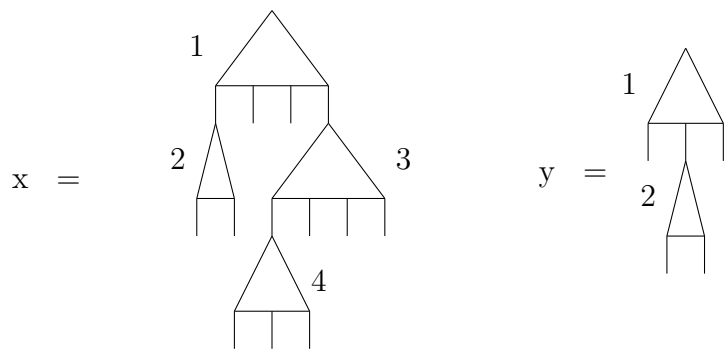
\begin{tikzpicture}

\draw (-2,3.5) node[above] {x\;\; =} ++(1,0);
\draw (0,5.25) node[above] {1} ++(1,0);
\draw (-0.25,3.75) node[above] {2} ++(1,0);
\draw (2.75,3.75) node[above] {3} ++(1,0);
\draw (1.625,2.125) node[above] {4} ++(1,0);

\draw (0.25,5) -- (1,6);
\draw (1,6) -- (1.75,5);
\draw (0.25,5) -- (1.75,5);
\draw (0.25,5) -- (0.25,4.5);
\draw (1.75,5) --(1.75, 4.5);
\draw (0.75,5) -- (0.75,4.5);
\draw (1.25,5) -- (1.25, 4.5);
\draw (0.25,4.5)--(0,3.5);
\draw (0.25,4.5)--(0.5,3.5);
\draw (0,3.5)--(0.5,3.5);
\draw (0,3.5) -- (0,3);
\draw (0.5,3.5) -- (0.5,3);
\draw (1.75,4.5) -- (1, 3.5);
\draw (1.75,4.5) -- (2.5,3.5);
\draw (1, 3.5) -- (2.5,3.5);
\draw (1,3.5) -- (1,3);
\draw (1.5,3.5) -- (1.5,3);
\draw (2,3.5) -- (2,3);
\draw (2.5,3.5) -- (2.5,3);
\draw (1,3) -- (0.5,2);
\draw (1,3) -- (1.5,2);
\draw (0.5,2) -- (1.5,2);
\draw (0.5,2) -- (0.5,1.5);
\draw (1,2) -- (1,1.5);
\draw (1.5,2) -- (1.5,1.5);

\draw (5,3.5) node[above] {y\;\; =} ++(1,0);
\draw (5.75,4.75) node[above] {1} ++(1,0);
\draw (6,3.25) node[above] {2} ++(1,0);

\draw (6,4.5) -- (6.5,5.5);
\draw (7,4.5) -- (6.5,5.5);
\draw (6,4.5) -- (7,4.5);
\draw (6, 4) -- (6,4.5);
\draw (6.5,4) -- (6.5,4.5);
\draw (7,4) -- (7,4.5);
\draw (6.5,4) -- (6.25,3);
\draw (6.5,4) -- (6.75,3);
\draw (6.25,3) -- (6.75,3);
\draw (6.25,3) -- (6.25,2.5);
\draw (6.75,3) -- (6.75,2.5);


\end{tikzpicture}
$$
(where the numbers denote the ordering of the entries of $x$ and $y$),
we can compose the two trees via $x\circ_3^2 y$ (see the below picture)
$$
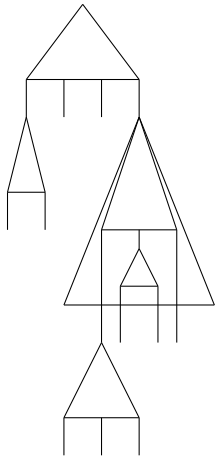
\begin{tikzpicture}
\draw (0.25,5) -- (1,6);
\draw (1,6) -- (1.75,5);
\draw (0.25,5) -- (1.75,5);
\draw (0.25,5) -- (0.25,4.5);
\draw (0.75,5) -- (0.75, 4.5);
\draw (1.25,5) -- (1.25, 4.5);
\draw (1.75,5) --(1.75, 4.5);
\draw (0.25,4.5)--(0,3.5);
\draw (0.25,4.5)--(0.5,3.5);
\draw (0,3.5)--(0.5,3.5);
\draw (0,3.5)--(0,3);
\draw (0.5,3.5)--(0.5,3);
\draw (0.75,2)--(1.75,4.5);
\draw (2.75,2)--(1.75,4.5);
\draw (0.75,2)--(2.75,2);
\draw (1.25,3)--(1.75,4.5);
\draw (1.75,4.5)--(2.25,3);
\draw (1.25,3)--(2.25,3);
\draw (1.25,3)--(1.25,1.5);
\draw (1.75,3)--(1.75,2.75);
\draw (2.25,3)--(2.25,1.5);
\draw (1.75,2.75)--(1.5, 2.25);
\draw (1.75,2.75)--(2,2.25);
\draw (1.5,2.25)--(2,2.25);
\draw (2,2.25)--(2,1.5);
\draw (1.5,2.25)--(1.5,1.5);
\draw (1.25,1.5)--(0.75,0.5);
\draw (1.75,0.5)--(1.25,1.5);
\draw (0.75,0.5)--(1.75,0.5);
\draw (0.75,0.5)--(0.75,0);
\draw (1.75,0.5)--(1.75,0);
\draw (1.25,0.5)--(1.25,0);
\end{tikzpicture}$$

\vspace{5mm}

The induction data of $B_n$ also includes increasing functions
$$\varphi_n^{(x,i,y)}:\{1,\dots,m_x\}\smallsetminus \{i\}\r \{1,\dots, m_x+m_y-1\}$$
$$\psi_n^{(x,i,y)}:\{1,\dots, m_y\}\r\{1,\dots,m_x+m_y-1\}$$
expressing how the elements of the two objects of $B_n$ involved in the
composition are shuffled in the resulting object.
The purpose of these functions is to formulate an analogue of the associativity axiom. 
For example, we see that for $n=1$, where the elements are prongs, the prongs of $y$ ``stay together"
in the shuffle even though the indexing is shifted. For $n=2$, instead of prongs, the elements are triangles. They are also ordered ``left to right" (and in case of a tie, ``top to bottom"), but
we see that in this ordering, both shuffles  $\varphi$ and $\psi$ are non-trivial, i.e. the ``triangles" of $y$ do not ``stay together" after the composition.

The associativity property for composition we require states that for $1\leq i<j\leq m_x$, $G_n(z)=f_j(x)$,
\beg{assocforcirc}{(x\circ^n_i y)\circ_{\varphi_n^{(x,i,y)}(j)}^n z= (x\circ_j^n z)\circ_i^n y.}

For $n=1$, we have
$$\varphi_1^{(x,i,y)}(j)=j+m_y-1 \text{ for }i<j\leq m_x$$
$$\psi_1^{(x,i,y)}(j)=j+i-1 \text{ for } 1 \leq j \leq m_y.$$

\vspace{1mm}

To make a rigorous definition of plain $n$-actads, we need to formulate a few more properties of composition and these functions.
For $j<i$, we must have
$$\varphi_n^{(x,i,y)}(j)=j.$$

(We shall sometimes omit the superscript $n$ when it is clear).
We will also have, for $i<j<k$ and any $n$, the additional axioms
\beg{varphiaxlong}{\varphi_n^{(x\circ_j^n z,i,y)}(\varphi_n^{(x,j,z)}(k))=\varphi_n^{(x\circ^n_i y,\varphi_n^{(x,i,y)}(j),z)}(\varphi_n^{(x,i,y)}(k)),}

\beg{varphiaxshort}{\varphi_n^{(x,i,y)}(j)=\varphi_n^{(x\circ_k t,i,y)}(j).}

\vspace{3mm}

Suppose we have defined $B_n$ (and all associated structure).

\begin{definition}\label{PlainActadDefinition}
A {\em plain $n$-actad} has the data of a set $\mathcal{C}$ with a map
$$\mathcal{C}\r B_n,$$
and, writing $\mathcal{C}(x)$ for the inverse image of and element $x\in B_n$,
composition operators
$$\gamma_{n,i}: \mathcal{C}(x)\times\mathcal{C}(y)\r \mathcal{C}(x\circ_i^n y)$$
for $x,y\in B_n$, with $1\leq i\leq m_x$, $G_n(y)=f_i(y)$.
We require that these composition operators satisfy the condition
that, for $1\leq i< j\leq m_x$, $z\in B_n$, $G_n(z)=f_j(x)$, the following diagram commutes:
\beg{assocgamma}{
\diagram
\mathcal{C}(x)\times\mathcal{C}(y)\times\mathcal{C}(z)\rto^{\gamma_{n,i}\times Id}\dto_{T}& \mathcal{C}(x\circ^n_i y)\times \mathcal{C}(z)\dto^{\gamma_{n, \varphi_n^{(x,i,y)}(j)}}\\
\mathcal{C}(x)\times\mathcal{C}(z)\times\mathcal{C}(y)\dto_{\gamma_{n,j}\times Id}& \mathcal{C}((x\circ_i^n y)\circ_{\varphi_n^{(x,i,y)}(j)}^n z)\dto^=\\
\mathcal{C}(x\circ_j^n z)\times\mathcal{C}(y) \rto^{\gamma_{n,i}}& \mathcal{C}((x\circ_j^n z)\circ_{i}^n y),\\
\enddiagram
}
where $T:\mathcal{C}(x)\times\mathcal{C}(y)\times\mathcal{C}(z)\r \mathcal{C}(x)\times\mathcal{C}(z)\times\mathcal{C}(y)$
is the permutation switching the coordinates of $\mathcal{C}(z)$ and $\mathcal{C}(y)$.
(We shall also sometimes write $\gamma_i$ instead of $\gamma_{n,i}$ when the value of $n$ is clear.)

\end{definition}

\vspace{3mm}

We shall now inductively define $B_n$.
For $n=1$, let $B_1=\N$. In addition, define
$$F_1:B_1\r \{(a_1,\dots,a_k)|a_i\in B_0=\{*\}\},$$
$$G_1:B_1\r B_0$$
by $F_1(n)=\underbrace{(*,*,\dots,*)}_{n \text{ times}}$, $G_1(n)=\{*\}$.
Then, for $x,y\in B_1$, let
$$x\circ^1_i y=x+y-1.$$
Note again that this works for every $i$ because the coordinates of $F_1(x)$ are all *, and $G_1(y)=*$ for every $y$.

\vspace{3mm}

Suppose we have $B_n$ and all the maps associated with it.
By Definition \ref{PlainActadDefinition} above, we have the notion of plain $n$-actads.
Given this, define
$$B_{n+1}=\{((x_1,x_2,\dots,x_k),(i_1,i_2,\dots i_{k-1}))\mid k\geq 1, x_\ell \in B_n,$$
$$G_n(x_j)=f_{i_{j-1}}((\dots(x_1\circ_{i_1}^n x_2)\circ_{i_2}^n x_3\dots )\circ_{i_{j-2}}^n x_{j-1}),$$
$$1\leq i_j\leq m_{(\dots(x_1\circ_{i_1}^n x_2)\circ_{i_2}^n x_3\dots )\circ_{i_{j-1}}^n x_j}=m_{x_1}+\dots+m_{x_{j}}-(j-1),$$
$$i_1\leq i_2\leq\dots \leq i_{k-1}\}.$$
In addition, define
$$G_{n+1}((x_1,x_2,\dots,x_k),(i_1,i_2,\dots, i_{k-1}))=$$
$$=(\dots ((x_1\circ_{i_1}^n x_2)\circ_{i_2}^n x_3)\dots)\circ_{i_{k-1}}^n x_k,$$

\vspace{3mm}

$$F_{n+1}((x_1,x_2\dots,x_k),(i_1,i_2,\dots,i_{k-1}))=$$
$$=(x_1,x_2,\dots,x_k),$$
and
$$ m_{((x_1,x_2,\dots,x_k),(i_1,i_2,\dots,i_{k-1}))}=k.$$

\vspace{5mm}

Then we have the following
\begin{lemma}\label{Lemmaplain}
$B_{n+1}$ is the free plain $n$-actad on $Id:B_n\r B_n$, where
$$((x_1,x_2,\dots, x_k),(i_1,i_2,\dots, i_{k-1}))=$$
$$\gamma_{n,i_{k-1}}(\gamma_{n,i_{k-2}} ( \dots \gamma_{n,i_2} (\gamma_{n,i_1}(x_1, x_2), x_3), x_4) \dots ), x_k).$$
\end{lemma}
\begin{proof}

\vspace{3mm}

The free plain $n$-actad on $B_n$ is
$$T_n=\{\gamma_{n,i_1}(\gamma_{n,i_2}(\dots\gamma_{n,i_k}(x,y_{k}),\dots,y_{2}),y_{1})|i_k,\dots,i_1\in \N,$$ 
$$ x, y_1, \dots , y_k \in B_n \text{ composable}\}/\sim$$
where $ \sim $ is the smallest equivalence relation
containing compositions of the $\gamma$'s differing only in one pair of consecutive 
$\gamma$'s, replaced by another according to Diagram \rref{assocgamma}.
We have, by definition,
$$B_{n+1}=\{\gamma_{n,j_1}(\gamma_{n,j_2}(\dots \gamma_{n,j_k}(x,x_{k}),\dots,x_{2}),x_{1})|j_k\leq\dots\leq j_1,$$
$$x, x_1, \dots , x_k \in B_n \text{ composable}\}.$$
We want
$$B_{n+1}=T_n.$$
We have a canonical map
\beg{canmapBT}{B_{n+1}\r T_n}
Suppose $i<j$. Then we have, by Diagram \rref{assocgamma},
$$\gamma_i(\gamma_j(x,z),y)=\gamma_{\varphi_n^{(x,i,y)}(j)}(\gamma_i(x,y),z).$$
and $\varphi_n^{(x,i,y)}(j)\geq j>i$. 
So, for every 
\beg{elementofBn+1}{\gamma_{n,i_1}(\gamma_{n,i_2}(\dots\gamma_{n,i_k}(x,y_{k}),\dots,y_{2}),y_{1})\in T_n,}
we can use $\varphi_n^{(x,i_{\ell},y)}(i_{\ell+1})$ to `switch' the order of $i_{\ell}$ and $i_{\ell'}$ whenever $i_{\ell}<i_{\ell+1}$.
So, there exist $j_1\geq\dots\geq j_k$ such that 
$$\gamma_{n,i_1}(\gamma_{n,i_2}(\dots\gamma_{n,i_k}(x,y_{k}),\dots,y_{2}),y_{1})=$$
$$=\gamma_{n,j_1}(\gamma_{n,j_2}(\dots \gamma_{n,j_k}(x,y_{\sigma(k)}),\dots,y_{\sigma(2)}),y_{\sigma(1)})$$
where $\sigma$ is a suitable permutation. (Though we note that it is not necessarily true that $\{i_1, \dots , i_k\}= \{j_1, \dots, j_k\}$.
This gives a left inverse of \rref{canmapBT}, so \rref{canmapBT} is onto.
However, then we need to know that we get the same answer regardless of the order of `switches'.
In proving this, the axioms \rref{varphiaxlong}, \rref{varphiaxshort} will come into play.

\vspace{3mm}

Suppose $i<j<k$. Performing switches in two different ways, we get
$$\gamma_i(\gamma_j(\gamma_k(x,t),z),y)=\gamma_{\varphi_n^{(\gamma_k(x,t),i,y)}(j)}(\gamma_i(\gamma_k(x,t)))=$$
$$=\gamma_{\varphi_n^{(\gamma_k(x,t),i,y)}(j)}(\gamma_{\varphi_n^{(x,i,y)}(k)}(\gamma_i(x,y),t),z)=$$
$$\gamma_{\varphi_n^{(\gamma_i(x,y),\varphi_n^{(\gamma_k(x,t),i,y)}(j),z)}(\varphi_n^{(x,i,y)}(k))}(\gamma_{\varphi_n^{(\gamma_k(x,t),i,y)}}(\gamma_i(x,y),z),t)$$
and
$$\gamma_i(\gamma_j(\gamma_k(x,t),z),y)=\gamma_i(\gamma_{\varphi_n^{(x,j,z)}(k)}(\gamma_j(x,z),t),y)=$$
$$=\gamma_{\varphi_n^{(\gamma_j(x,z),i,y)}(\varphi_n^{(x,j,z)}(k))}(\gamma_i(\gamma_j(x,z),y),t)=$$
$$=\gamma_{\varphi_n^{(\gamma_j(x,z),i,y)}(\varphi_n^{(x,j,z)}(k))}(\gamma_{\varphi_n^{(x,i,y)}(j)}(\gamma_i(x,y),z),t).$$
Thus, we must have
\beg{longphiaxiom}{\varphi_n^{(\gamma_j(x,z),i,y)}(\varphi_n^{(x,j,z)}(k))=\varphi_n^{(\gamma_i(x,y),\varphi_n^{(\gamma_k(x,t),i,y)}(j),z)}(\varphi_n^{(x,i,y)}(k)),}
\beg{shortphiaxiom}{\varphi_n^{(x,i,y)}(j)=\varphi_n^{(\gamma_k(x,t),i,y)}(j),}
which are \rref{varphiaxlong} and \rref{varphiaxshort}.
Conversely, \rref{varphiaxlong} and \rref{varphiaxshort} guarantee that two switches have the same result, i.e.
\beg{longshortaxiom}{\varphi_n^{(\gamma_j(x,z),i,y)}(\varphi_n^{(x,j,y)}(k))=\varphi_n^{(\gamma_i(x,y),\varphi_n^{(x,i,y)}(j),z)}(\varphi_n^{(x,i,y)}(k)).}

\vspace{3mm}

Now, we will use induction to prove that we always get the same result regardless of the order we perform switches
of consecutive pairs $\gamma_i\gamma_j$, for $i<j$. Let
$$i=min\{i_1,\dots,i_k\},$$
where $i_1,\dots,i_k$ are as in \rref{elementofBn+1}. If $i$ occurs only once in the sequence $(i_1,\dots,i_k)$, then let
$n_i$ be the distance from $i$ to the end of the sequence.
If there exist $\ell_1<\dots<\ell_m\in \{1,\dots,k\}$ such that
$$i=i_{\ell_1},\dots,i_{\ell_m},$$
then let $n_i$ be the distance from $i_{\ell_m}$ to the end of the sequence. 

If $i$ is at the end of the sequence (i.e. $i=i_k$ and $n_i=0$), then repeat this process for $(i_1,\dots,i_{k-1})$. If the smallest $i$ in every sequence
is at the end, $(i_1,\dots,i_k)$ is already in order. So, then we are done.
 
Suppose we know that the result is independent of the order we take the switches if $n_i=p$.
Then suppose $n_i=p+1$. 

The strategy is to show that without loss of generality, all the possible switches from $j<k$ to $j>k$ to the right of $i$
can be done first before $i$ is moved to the right. Then the order of the switches to the right of $i$ does not matter by the
induction hypothesis.

Consider the position when the last switch to the right of $i$ before $i$ is moved to the right has been performed.
Assume first that the situation immediately to the right of $i$ is
\beg{iljgk}{i<j>k.}
Thus, the next position will be 
\beg{jpgilk}{j'>i<k.}
Now if any more switches to the right of $i$ happen before $i$ is moved again, then equivalently, those switches could have been
made first, and then the move from \rref{iljgk} to \rref{jpgilk}.

Thus, either all the possible switches to the right of $i$ have been performed before \rref{iljgk} was reached (which is what we were trying
to assume), or the position right before $i$ was moved to the right was
\beg{iljlk}{i<j<k.}
Now $i$ is moved to the right producing
$$\bar{j}>i<k.$$
By the induction hypothesis, the order of the remaining moves does not matter,
so we may as well continue to
$$\bar{j}<\bar{k}>i,$$
$$\bar{k}>\bar{j}>i.$$
But by axioms \rref{varphiaxlong} and \rref{varphiaxshort}, those are equivalent to
$$i<k'>j$$
\beg{tripleijk}{k''>i<j}
$$k''>j''>i$$
In the sequence \rref{tripleijk}, one more move was executed to the right of $i$ before $i$ was
moved.

\vspace{3mm}

\end{proof}

Note that this proof is analogous to the proof of the presentation of $\Sigma_n$ by the ``Yang-Baxter relations"
$a_ia_{i+1}a_i=a_{i+1}a_ia_{i+1}$ (along with $a_ia_j=a_ja_i$ for $j>i+1$) where $a_i=(i,i+1)$ in the cycle notation.

\vspace{5mm}

Then, the functions $\varphi_{n+1}^{(x,j,y)}$, $\psi_{n+1}^{(x,j,y)}$ for
$$x=((x_1,x_2,\dots,x_k),(i_1,i_2,\dots,i_{k-1}))$$
$$y=((y_1,y_2,\dots,y_{\ell}),(\iota_1,\iota_2,\dots,\iota_{\ell-1})).$$
are determined as follows:

\noindent
For a plain $n$-actad $\mathcal{C}$, and for $\zeta_i\in\mathcal{C}(x_i)$, $\zeta_{\iota}\in\mathcal{C}(y_{\iota}),$
we have

$$\gamma_{n,i_{k-1}}(\dots \gamma_{n,i_j}(\gamma_{n,i_{j-1}}(\dots \gamma_{n,i_2}(\gamma_{n,i_1}(\zeta_1,\zeta_2),\zeta_3)\dots,\zeta_{j-1}),$$
$$\gamma_{n,\iota_{\ell-1}}(\dots \gamma_{n,\iota_2}(\gamma_{n,\iota_1}(\zeta_1,\zeta_2),\zeta_3),\dots,\zeta_\ell),\dots x_k)=$$
$$\gamma_{n,\lambda_{k+\ell-1}}(\dots \gamma_{n,\lambda_2}(\gamma_{n,\lambda_1}(\mu_1,\mu_2),\mu_3)\dots, \mu_{\lambda_{k+\ell-1}}),$$
with 
$$\lambda_1\leq\dots\leq \lambda_{k+\ell-1}$$
$$\zeta_s=\mu_{\varphi_{n+1}^{(x,j,y)}(s)}$$
$$\zeta_t=\mu_{\psi_{n+1}^{(x,j,y)}(t)},$$
for $s\in\{1,\dots,k\},t\in\{1,\dots,\ell\}$.
Finally, also define
\beg{circinduction}{\begin{array}{c}
((x_1,\dots,x_k),(i_1,\dots i_{k-1}))\circ_j^{n+1}((y_1,\dots, y_{\ell}),(\iota_1,\dots,\iota_{\ell-1}))=\\
=\gamma_{n,i_{k-1}}(\dots \gamma_{n,i_j}(\gamma_{n,i_{j-1}}(((x_1,\dots,x_k),(i_1,\dots i_{k-1})),\\
((y_1,\dots, y_{\ell}),(\iota_1,\dots,\iota_{\ell-1}))),x_{j+1}),\dots,x_k).\\
\end{array}
}

Axioms \rref{assocforcirc} and \rref{varphiaxlong} then follow from the fact that one can similarly define, directly analogously,
a composition of $((x_1,\dots,x_k),(i_1,\dots,i_{k-1}))$ with two appropriate
elements 
$$((y_1,\dots, y_{\ell_1}),(\iota_1,\dots,\iota_{\ell_1-1}))$$ and 
$$((z_1,\dots,z_{\ell_2}),(\kappa_1,\dots,\kappa_{\ell_2-1}))$$ at
$1\leq j_1<j_2\leq m_x$, and observe that it is equal, by definition, to the compositions in the two different orders, and similarly
for the shuffles.

Axiom \rref{varphiaxshort} follows from the fact that, by Lemma \ref{Lemmaplain}, composing at a $j$, as in \rref{circinduction} does not affect the 
shuffles of compositions at $j'<j$.

\vspace{5mm}

\section{Actads and Bases} \label{s2}

\vspace{5mm}

Using this, we will define the $n$-base $\mathscr{B}_n$, and then $n$-actads.
Unlike the $B_n$'s, the $\mathscr{B}_n$'s are not sets. $\mathscr{B}_n$ and $n$-actads are $n$-fold categories.
(Recall, these are more general than $n$-categories. For example, $n=2$ gives the notion of bi-categories, see \cite{Bicategory},
which strictly contains the notion of $2$-categories.)
One can define an $n$-fold category inductively as a category internal in $(n-1)$-fold categories
(i.e. where both objects and morphisms are $(n-1)$-fold categories, and the maps
involved in the axioms of the category are morphisms of $(n-1)$-fold categories, i.e. $(n-1)$-fold functors).
We start the induction with $1$-fold categories. Note that $n$-fold categories still have pullbacks.

This means that it involves sets of ``$S$-morphisms" for $S$ any subset of $\{1,\dots, n\}$.
In particular, ``$\emptyset$-morphisms" are actually the objects, which for $\mathscr{B}_n$ are
the plain $n$-bases, and for $n$-actads, they are the plain $n$-actads.
To discuss this further, let us introduce some notation:

\vspace{3mm}

Suppose $\mathscr{C}$ is an $n$-fold category and $S\subseteq\{1,\dots,n\}$.
For 
\beg{ISto01}{\mathbb{I}:S\r \{0,1\},}
we have the source-target map
\beg{scrSI}{\mathscr{S}_{\mathbb{I}}: S-Mor(\mathscr{C})\r (\emptyset-Mor(\mathscr{C}))=: Obj(\mathscr{C}).}
(Then $0$ stands for source and $1$ stands for target.)  
In addition, let $\underline{0}:S\r\{0,1\}$ be the constant $0$ function.

\vspace{5mm}

More generally, for $T\subseteq S$, $\mathbb{I}:S\smallsetminus T\r\{0,1\}$, we have a
source-target map
$$\mathscr{S}_{\mathbb{I}}^T:S-Mor(\mathscr{C})\r T-Mor(\mathscr{C}).$$
Call an $n$-fold category $\mathscr{C}$ {\em cube-like} if for all $S\subseteq\{1,\dots,n\}$, for all $x\in Obj(\mathscr{C})$, and
$$\phi_i\in \{i\}-Mor(\mathscr{C})$$
(for any $i\in S$), such that $\mathscr{S}_{\underline{0}}(\phi_i)=x$, there exists a unique $\phi\in S-Mor(\mathscr{C})$ such that
$$\mathscr{S}_{\underline{0}}^{\{i\}}(\phi)=\phi_i.$$

\vspace{3mm}

Being cube-like is a very substantial restriction on $n$-fold categories, essentially still a concept of ($1$-fold) categories, since
the condition only involves morphisms between objects.
In this paper, all $n$-fold categories we will deal with will be cube-like, unless we specify otherwise. Also, in all of our examples of cube-like $n$-fold
categories, the $1$-subcategories of $\{i\}$-morphisms will be groupoids. We shall call such cube-like categories 
{\em invertible cube-like}. In particular, $\mathscr{B}_n$ and the $n$-actads will be invertible cube-like $n$-fold categories.
An $n$-fold functor $F:\mathscr{C}\r\mathscr{D}$ is called {\em fibered} if for every $x\in Obj(\mathscr{C})$, for
every $\phi\in S-Mor(\mathscr{D})$ with $\mathscr{S}_{\underline{0}} (\phi)=F(x)$,
there exists a unique $\psi\in S-Mor(\mathscr{C})$ such that $\mathscr{S}_{\underline{0}}(\psi)=x$, $F(\psi)=\phi$,
and all $S$-morphisms of $\mathscr{C}$ are given this way.
(Note that other variants of this concept exist).

Also, define

\begin{definition}
Define the category $I$ by 
$$Obj(I)=\{0,1\}$$
$$Mor(I)=\{(0\r 0),(0\r 1),(1\r 1)\}.$$
Then let $Set_{(n)}$ be the $n$-fold category of sets, defined as follows: For a subset $S\subseteq\{ 1, \dots , n\}$,
let the set of $S$-morphisms is
$$S-Mor(Set_{(n)})=\{F| F \text{ is a functor } I^S\r Set\}.$$
\end{definition}

Now, we have a Lemma:

\begin{lemma}

\vspace{1.5mm}

\begin{enumerate}

\vspace{1.5mm}

\item\label{lemmaitem1}

Let $\mathscr{D}$ be an $n$-fold category. A fibered $n$-fold functor $F:\mathscr{C}\r\mathscr{D}$ produces (in a bijective fashion)
an $n$-fold functor $\Phi:\mathscr{D}\r Set_{(n)}$: for some $S \subseteq \{1, \dots , n\}$, $\psi \in S-Mor (\mathscr{D})$,
\beg{Lemmaitemnontriv}{\Phi(\psi)(\mathbb{I})=\mathscr{S}_{\mathbb{I}}(F^{-1}(\psi)),}
(recall \rref{ISto01} for the definition of $\mathbb{I}$) with
$$\Phi(\psi):Mor(I^{|S|})\r Mor(Set)$$
defined by the fibered property of $F$.

\vspace{3mm}

\item\label{lemmaitem2}

If $F:\mathscr{C}\r\mathscr{D}$ is a fibered $n$-fold functor and $\mathscr{D}$ is cube-like, then $\mathscr{C}$ is cube-like.

\end{enumerate}
\end{lemma}
\begin{proof}

\vspace{3mm}

First we will discuss $($\ref{lemmaitem1}$)$.
Starting with a fibered $n$-fold functor $F:\mathscr{C}\r\mathscr{D}$, define $\Phi$ by \rref{Lemmaitemnontriv}. Then the fact that $\Phi$ is an $n$-fold functor
follows directly from the properties of composition in an $n$-fold category.

Given an $n$-fold functor $\Phi:\mathscr{D}\r Set_{(n)}$, define, for $S\subseteq \{1,\dots,n\}$,
\beg{proofitemnontriv}{S-Mor(\mathscr{C})=\coprod_{\psi\in S-Mor(\mathscr{D})} \Phi(\psi)}
with $F:\mathscr{C}\r\mathscr{D}$ given by projection.
Then $F$ is a fibered $n$-fold functor by the functoriality properties of $F$. Additionally,
the constructions of $\Phi$ from $F$ and $F$ from $\Phi$ are obviously inverse to each other.

\vspace{3mm}

$($\ref{lemmaitem2}$)$ is immediate from \rref{proofitemnontriv}, since in a fibered $n$-fold functor,
for $S\subseteq \{1,\dots,n\}$, we can lift the $\{i\}$-morphisms with $i\in S$ with a common source uniquely to $\{i\}$-morphisms
with a given common source, which then complete uniquely to an $S$-morphism both downstairs and upstairs.

\end{proof}

\vspace{5mm}

We will also need a way to construct an $n$-fold category from sequences of morphisms of an $(n-1)$-fold category for the induction
step from $\mathscr{B}_n$ to $\mathscr{B}_{n+1}$. (This is a generalization of $T_{B_{n-1}}$ from the case of plain $n$-actads and $n$-bases.)

\begin{definition}\label{defnTCinduced}
Suppose $C$ is an $(n-1)$-fold category.
Define the $n$-fold category $T_C$ in the following way:
Suppose $S\subseteq \{1,\dots,n-1\}$. Then define the $S$-morphisms by
$$
S-Mor(T_C)=\{(\alpha_1,\dots,\alpha_k)|k \geq 1,\; \alpha_i\in S-Mor(C)\}$$
\beg{TCinducednmors}{\begin{array}{c}
(S\cup \{n\})-Mor(T_C)=\\
\{f=((\alpha_1,\dots,\alpha_k),\sigma
,(\beta_1,\dots, \beta_k))\\
:(\alpha_1,\dots,\alpha_k)\r(\beta_1,\dots,\beta_k)\mid k\geq 1, \; \alpha_i\in S-Mor(C)\\
\sigma:\{1,\dots,k\}\r\{1,\dots,k\}\text{ permutation}, \forall i\\
\beta_i=\alpha_{\sigma(i)}\}.\\
\end{array}}
\end{definition}

\vspace{5mm}

For $f\in S\cup\{n\}-Mor(T_C)$ in \rref{TCinducednmors}, we will sometimes denote the permutation $\sigma$
by $\sigma_f$.
It suffices to have $\mathscr{B}_n$ as an $(n-1)$-fold category, with
an $(n-1)$-fold functor
$$\mathscr{F}_n:\mathscr{B}_n\r T_{\mathscr{B}_{n-1}}.$$
Then the $n$-fold category structure of $\mathscr{B}_n$ is induced by $\mathscr{F}_n$.
Explicitly, for $S\subseteq \{1,\dots,n-1\}$, $x,y\in S-Mor(\mathscr{B}_n)$,

\vspace{0.5mm}

$$(S\cup\{n\})-Mor(\mathscr{B}_n)(x,y)=(S\cup\{n\})-Mor(T_{\mathscr{B}_{n-1}})(\mathscr{F_n}(x),\mathscr{F_n}(y)).$$
 
\vspace{0.5mm}

Then for $S\subseteq \{1,\dots, n\}$, $x\in S-Mor(\mathscr{B}_n)$, denote again
the length of $\mathscr{F}_n(x)$ by $m_x$ and
$$\mathscr{F}_n(x)=:(f_1(x),\dots,f_{m_x}(x)).$$
Thereby, $\mathscr{F}_n$ becomes an $n$-fold functor. We will also have an $(n-1)$-fold functor
$$\mathscr{G}_n:\mathscr{B}_n\r\mathscr{B}_{n-1}$$
given by applying the $\mathscr{B}_{n-1}$-composition, similarly as in our definition of $G_n$.

\vspace{3mm}

An example of a $\{1\}$-morphism in $\mathscr{B}_2$ is the map that takes the tree on the left to the tree on the right by

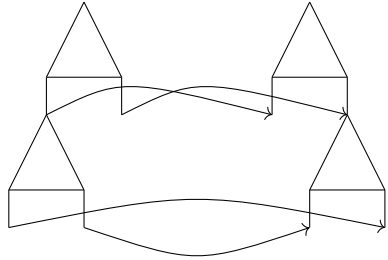
\beg{1permutationof2trees}{\begin{tikzpicture}
\draw(0,0)--(0,0.5);
\draw(0,0.5)--(1,0.5);
\draw(1,0)--(1,0.5);
\draw(0,0.5)--(0.5,1.5);
\draw(1,0.5)--(0.5,1.5);
\draw(0.5,1.5)--(0.5,2);
\draw(0.5,2)--(1.5,2);
\draw(1.5,2)--(1.5,1.5);
\draw(0.5,2)--(1,3);
\draw(1,3)--(1.5,2);

\draw(5,0)--(5,0.5);
\draw(5,0.5)--(4,0.5);
\draw(4,0)--(4,0.5);
\draw(5,0.5)--(4.5,1.5);
\draw(4,0.5)--(4.5,1.5);
\draw(3.5,1.5)--(3.5,2);
\draw(3.5,2)--(4.5,2);
\draw(4.5,2)--(4.5,1.5);
\draw(3.5,2)--(4,3);
\draw(4,3)--(4.5,2);
\draw [->] (0, 0) .. controls (2.5, 0.5) .. (5, 0);
\draw [->] (1, 0) .. controls (2.5, -0.5) .. (4, 0);
\draw [->] (0.5, 1.5) .. controls (1.5, 2) .. (3.5, 1.5);
\draw [->] (1.5, 1.5) .. controls (2.5, 2) .. (4.5, 1.5);
\end{tikzpicture}}


\noindent Call it $\vartheta$.
An example of a $\{2\}$-morphism in $\mathscr{B}_2$ is the map that takes the tree on the left to the tree on the right by

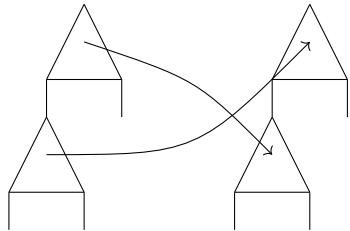
\beg{2permutationof2trees}{\begin{tikzpicture}
\draw(0,0)--(0,0.5);
\draw(0,0.5)--(1,0.5);
\draw(1,0)--(1,0.5);
\draw(0,0.5)--(0.5,1.5);
\draw(1,0.5)--(0.5,1.5);
\draw(0.5,1.5)--(0.5,2);
\draw(0.5,2)--(1.5,2);
\draw(1.5,2)--(1.5,1.5);
\draw(0.5,2)--(1,3);
\draw(1,3)--(1.5,2);

\draw(3,0)--(3,0.5);
\draw(3,0.5)--(4,0.5);
\draw(4,0)--(4,0.5);
\draw(3,0.5)--(3.5,1.5);
\draw(4,0.5)--(3.5,1.5);
\draw(3.5,1.5)--(3.5,2);
\draw(3.5,2)--(4.5,2);
\draw(4.5,2)--(4.5,1.5);
\draw(3.5,2)--(4,3);
\draw(4,3)--(4.5,2);
\draw [->] (0.5, 1) .. controls (2.5, 1) .. (4, 2.5);
\draw [->] (1, 2.5) .. controls (2.5, 2) .. (3.5, 1);
\end{tikzpicture}}


\noindent Call it $\varrho$. Then we have a $\{1,2\}$-morphism in $\mathscr{B}_2$ given by the following diagram:
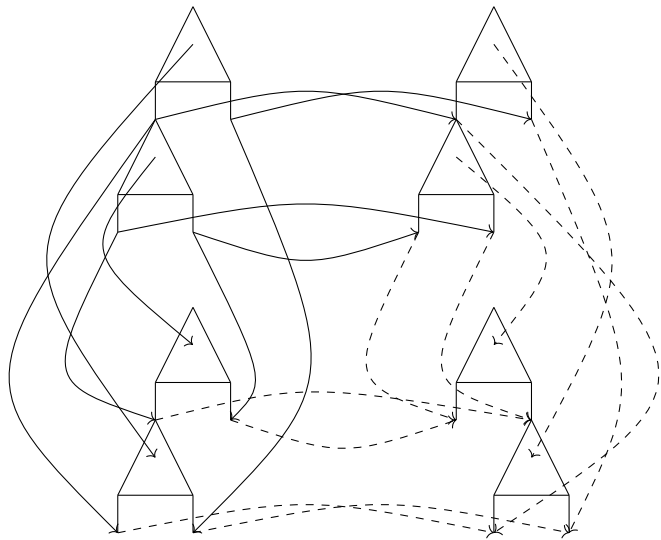
\vspace{3mm}

\hspace{15mm}\begin{tikzpicture}
\draw(0,0)--(0,0.5);
\draw(0,0.5)--(1,0.5);
\draw(1,0)--(1,0.5);
\draw(0,0.5)--(0.5,1.5);
\draw(1,0.5)--(0.5,1.5);
\draw(0.5,1.5)--(0.5,2);
\draw(0.5,2)--(1.5,2);
\draw(1.5,2)--(1.5,1.5);
\draw(0.5,2)--(1,3);
\draw(1,3)--(1.5,2);

\draw(6,0)--(6,0.5);
\draw(6,0.5)--(5,0.5);
\draw(5,0)--(5,0.5);
\draw(6,0.5)--(5.5,1.5);
\draw(5,0.5)--(5.5,1.5);
\draw(4.5,1.5)--(4.5,2);
\draw(4.5,2)--(5.5,2);
\draw(5.5,2)--(5.5,1.5);
\draw(4.5,2)--(5,3);
\draw(5,3)--(5.5,2);

\draw(0,4)--(0,4.5);
\draw(0,4.5)--(1,4.5);
\draw(1,4)--(1,4.5);
\draw(0,4.5)--(0.5,5.5);
\draw(1,4.5)--(0.5,5.5);
\draw(0.5,5.5)--(0.5,6);
\draw(0.5,6)--(1.5,6);
\draw(1.5,6)--(1.5,5.5);
\draw(0.5,6)--(1,7);
\draw(1,7)--(1.5,6);

\draw(4,4)--(4,4.5);
\draw(4,4.5)--(5,4.5);
\draw(5,4)--(5,4.5);
\draw(4,4.5)--(4.5,5.5);
\draw(5,4.5)--(4.5,5.5);
\draw(4.5,5.5)--(4.5,6);
\draw(4.5,6)--(5.5,6);
\draw(5.5,6)--(5.5,5.5);
\draw(4.5,6)--(5,7);
\draw(5,7)--(5.5,6);

\draw [dashed,->] (0, 0) .. controls (2.5, 0.5) .. (5, 0); 
\draw [dashed,->] (1, 0) .. controls (3.5, 0.5) .. (6, 0);
\draw [dashed,->] (0.5, 1.5) .. controls (2.5, 2) .. (5.5, 1.5);
\draw [dashed,->] (1.5, 1.5) .. controls (3, 1) .. (4.5, 1.5);
\draw [->] (1,6.5) .. controls (-1.5,3.75) .. (0.5,1);
\draw [->] (0.5,5) .. controls (-0.5,3.75) .. (1,2.5);
\draw [->] (0,4) .. controls (2.5,4.5) .. (5,4);
\draw [->] (1,4) .. controls (2.5,3.5) .. (4,4);
\draw [->] (0.5,5.5) .. controls (2.5,6) .. (4.5,5.5);
\draw [->] (1.5,5.5) .. controls (3.5,6) .. (5.5,5.5);
\draw [dashed,->] (4.5,5) .. controls (6,3.75) .. (5,2.5);
\draw [dashed,->] (5,6.5) .. controls (7,3.75) .. (5.5,1);
\draw [->] (0,4) .. controls (-1,2) .. (0.5,1.5);
\draw [->] (0.5,5.5) .. controls (-2,2) .. (0,0);
\draw [->] (1,4) .. controls (2,2) .. (1.5,1.5);
\draw [->] (1.5,5.5) .. controls (3,2) .. (1,0);
\draw [dashed,->] (5.5,5.5) .. controls (7,2) .. (6,0);
\draw [dashed,->] (4.5,5.5) .. controls (8,2) .. (5,0);
\draw [dashed,->] (4,4) .. controls (3,2) .. (4.5, 1.5);
\draw [dashed,->] (5,4) .. controls (4,2) .. (5.5,1.5);
\end{tikzpicture}

\vspace{3mm}

\noindent where the maps on the left are $\varrho$, and the maps on the bottom are $\vartheta$.

To capture the properties of the composition, define cube-like $n$-fold categories $Comp_n$, $Comp^2_n$ as follows:
For $S\subseteq\{1,\dots,n-1\}$, define
\beg{CompnS-morsdefn}{
\begin{array}{c}
S-Mor(Comp_n)=\\
\{(x,i,y)\mid x,y \in S-Mor(\mathscr{B}_n), 1\leq i\leq m_x, \mathscr{G}_n(y)=f_i(x)\},
\end{array}}

\vspace{1.5mm}

$$(S\cup\{n\})-Mor(Comp_n)=\{((x,i,y)\r(x,i',y), f, g)\mid1\leq i\leq m_x,$$
$$x,y\in S-Mor(\mathscr{B}_n), \mathscr{G}_n(y)=f_i(x),$$
$$f:x\r x,\; g:y\r y,\; f,g\in (S\cup\{n\})-Mor(\mathscr{B}_n), i'=\sigma_f(i)\},$$

\vspace{3mm}

$$S-Mor(Comp^2_n)=\{(x,i,j,y,z)\mid x,y,z\in S-Mor(\mathscr{B}_n),$$
$$1\leq i\neq j\leq m_x, \mathscr{G}_n(y)=f_i(x), \mathscr{G}_n (z)=f_j(x)\},$$

\vspace{1.5mm}

$$(S\cup\{n\})-Mor(Comp_n^2)=\{((x,i,j,y,z)\r(x,i',j',y,z),f,g,h)\mid$$
$$1\leq i\neq j\leq m_x,$$
$$x,y,z\in S-Mor(\mathscr{B}_n), \mathscr{G}_n(y)=f_i(x), \mathscr{G}_n(z)=f_j(x),$$
$$f:x\r x,\; g:y\r y,\; h:z\r z,\; f,g,h\in S\cup\{n\}-Mor(\mathscr{B}_n),$$
$$i'=\sigma_f(i), j'=\sigma_f(j)\}.$$

\vspace{3mm}

Note that we then have, by definition, projection $n$-fold functors
$$Comp_n\r \mathscr{B}_n\times\mathscr{B}_n,$$
$$Comp_n^2\r \mathscr{B}_n\times\mathscr{B}_n\times\mathscr{B}_n.$$

\vspace{3mm}

We are also given (as a part of the induction data) increasing functions 
$$\varphi_n^{(x,i,y)}:\{1,\dots,m_x\}\r \{1,\dots, m_x+m_y-1\},$$
$$\psi_n^{(x,i,y)} : \{1, \dots , m_x\} \r \{ 1, \dots , m_x+m_y-1\}$$
for $(x,i,y)\in S-Mor(Comp_n)$ for some $S\subseteq\{1,\dots, n-1\}$,
and a composition $n$-fold functor
\beg{PhiCompnBn}{\Phi_n:Comp_n\r \mathscr{B}_n.}
Then define
$$\Phi_n^1,\Phi_n^2:Comp_n^2\r Comp_n,$$
with 
$$\Phi_n^1(x,i,j,y,z)=(\Phi_n(x,i,y),\varphi_n^{(x,i,y)}(j),z)$$
$$\Phi_n^2(x,i,j,y,z)=(\Phi_n(x,j,y),\varphi_n^{(x,i,y)}(i),z).$$

\vspace{3mm}

On $(S\cup\{n\})-Mor(Comp_n^2)$, $\Phi_n$, $\Phi_n^1$, and $\Phi_n^2$ are defined by the wreath products of the permutations such that
\beg{compdiagram}{
\diagram
Comp_n^2\rto^{\Phi_n^1}\dto_{\Phi_n^2}& Comp_n\dto^{\Phi_n}\\
Comp_n\rto_{\Phi_n}& \mathscr{B}_n\\
\enddiagram
}
commutes strictly.
(Note that then the $\{n \}$-morphisms of $\mathscr{B}_n$ are permutations.)

\vspace{5mm}

We will construct the $\mathscr{B}_n$'s inductively.
Suppose we are given $\mathscr{B}_n$ and all of the associated maps.
Then take a cube-like $n$-fold category $\mathscr{C}$ with a fibered $n$-fold functor
$$\mathscr{C}\r \mathscr{B}_n.$$
For $S\subseteq\{1,\dots,n\}$, denote the preimage of an $S$-morphism $x$ by $\mathscr{C}(x)$.
Then let
$$Comp_{\mathscr{C}}=(\mathscr{C}\times\mathscr{C})\times_{(\mathscr{B}_n\times\mathscr{B}_n)}Comp_n$$
$$Comp_{\mathscr{C}}^2=(\mathscr{C}\times\mathscr{C}\times\mathscr{C})\times_{(\mathscr{B}_n\times\mathscr{B}_n\times\mathscr{B}_n)}
Comp_n^2.$$

\vspace{3mm}

\begin{definition}
A cube-like $n$-fold category $\mathscr{C}$ with a fibered $n$-fold functor $\mathscr{C} \rightarrow \mathscr{B}_n$
is called an {\em $n$-actad} provided that there exist $n$-fold functors
$$\Gamma:Comp_{\mathscr{C}}\r \mathscr{C},$$
$$\Gamma^1,\Gamma^2: Comp_{\mathscr{C}}^2\r Comp_{\mathscr{C}},$$
such that the following axioms hold:
\begin{enumerate}
\label{actadcond1}\item
$$\Gamma^1=(\Gamma\times Id_{\mathscr{C}})\times_{Comp_n\times\mathscr{B}_n}\Phi_n^1$$

\vspace{3mm}

\label{actadcond2}\item
$$\Gamma^2=(Id_{\mathscr{C}}\times\Gamma)\times_{\mathscr{B}_n\times Comp_n}\Phi_n^2$$

\vspace{3mm}

\label{actadcond3}\item
$$
\diagram
Comp_{\mathscr{C}}^2\rto^{\Gamma^1}\dto_{\Gamma^2}& Comp_{\mathscr{C}}\dto^{\Gamma}\\
Comp_{\mathscr{C}}\rto_{\Gamma}& \mathscr{C}\\
\enddiagram
$$
strictly commutes.
\end{enumerate}

\end{definition}

\vspace{5mm}

Let $\mathscr{C}_n$ be the category of $n$-actads, and let $\mathscr{D}_n$ be the category of $n$-fold fibered
categories over $\mathscr{B}_n$. Then we have a forgetful fibered functor
$$\mathscr{C}_n\r\mathscr{D}_n.$$
Then the left adjoint to this functor is called a free $n$-actad.

\vspace{5mm}

Suppose that we are given $\mathscr{B}_n$, the $n$-actads and all of the associated maps.

\vspace{3mm}

Suppose $\mathscr{C}$ is an $n$-fold category and $S\subseteq\{1,\dots,n\}$.
Recall that for $\mathbb{I}:S\r \{0,1\}$, we have the source-target map \rref{scrSI}.

We then define the $(n+1)$-base $\mathscr{B}_{n+1}$ as a cube-like $n$-fold category by letting, for $S\subseteq\{1,\dots, n\}$,
$S-Mor(\mathscr{B}_{n+1})$ be the free plain $n$-actad on
$$\mathscr{S}_{\underline{0}}:S-Mor(\mathscr{B}_n)\r B_n,$$
($B_n= Obj(\mathscr{B}_n)$).
One checks readily that this defines an invertible cube-like $n$-fold category. (In fact, because of the invertibility,
$\underline{0}$ could be equivalently replaced by any $\mathbb{I}: S\r\{0,1\}$.)

$\mathscr{B}_{n+1}$ then becomes an invertible cube-like $(n+1)$-fold category by the device described under Definition
\ref{defnTCinduced}, and satisfies \rref{compdiagram} with $n$ replaced by $n+1$.
Then $\Phi_{n+1}$ (see \rref{PhiCompnBn}) on $S$-morphisms with $S\subseteq\{1,\dots, n\}$ is defined by plain $n$-actad
composition of $S$-morphisms in $\mathscr{B}_n$. Again, $\Phi_{n+1}$ on $S\cup\{n+1\}$ are defined by wreath products of permutations (see
the comments under \rref{PhiCompnBn}).

It is tempting to conclude that $\mathscr{B}_{n+1}$ is the free $n$-actad on $\mathscr{B}_n$, but that is false because
for $S\subseteq\{1,\dots, n\}$, not all $S$-morphisms of $\mathscr{B}_n$ can be lifted to $\mathscr{B}_{n+1}$ via
$\mathscr{G}_{n+1}$.
This is due to the fact that additional structure is being introduced, where not all isomorphisms downstairs
will preserve it: for example, for $n=2$, not every $1$-permutation of the prongs of a planar tree
comes from an actual isomorphism of trees.
Thus, $\mathscr{B}_{n+1}$ is, in fact, not a fibered $n$-fold category over $\mathscr{B}_n$
(via $\mathscr{G}_n$), and consequently not an $n$-actad. On the other hand,
$\mathscr{G}_{n+1}:\mathscr{B}_{n+1}\r\mathscr{B}_n$ satisfies
the {\em uniqueness} axiom of a fibered $n$-fold category, and in fact,
the following is true (see also Section \ref{sMonads} below):

\begin{lemma}\label{KanExtensionLemmas2}
Let $\mathscr{U}$ be the forgetful functor from the category
of fibered $n$-fold categories over $\mathscr{B}_n$ to the category of $n$-fold categories over $\mathscr{B}_n$ (and $n$-fold functors).
Let $\mathscr{V}$ be the left adjoint to $\mathscr{U}$ (``left Kan extension"). Then $\bar{\mathscr{B}_{n+1}}=\mathscr{V}\mathscr{B}_{n+1}$ is the free $n$-actad
on $\mathscr{B}_n$ on $Id:\mathscr{B}_n\r\mathscr{B}_n$.
\end{lemma}
\qed

(Note that in the case of $n=1$, $1$-actads are simply operads.)

\vspace{5mm}

\section{Algebras over an $n$-Actad}\label{salg}

\vspace{5mm}

The next thing we will discuss is algebras over an $n$-actad $\mathscr{C}$.
For operads, which are $1$-actads, an algebra over an operad $\mathscr{C}$
is an $X$ with
$$\mathscr{C}(n)\times X^n\r X$$
satisfying the appropriate axioms (associativity, and, if we choose, commutativity in the broader sense
and unitality),
and a module over an operad $\mathscr{C}$ and an algebra $X$
is a $Y$ with
$$\mathscr{C}(n)\times X^{n-1}\times Y\r Y$$
satisfying similar axioms.

\vspace{3mm}

Now, take an $(n-1)$-fold category $X$.
Suppose we also have a fibered $(n-1)$-fold functor
$$\Xi:X\r \mathscr{B}_{n-1}$$
(then, in particular, $X$ is cube-like).
For $S\subseteq\{1,\dots,n-1\}$, $a\in S-Mor(\mathscr{B}_{n-1})$,
let the inverse image of $\Xi$ of $a$ again be called $X(a)$.

Suppose we have $S\subseteq \{1,\dots,n-1\}$.
Then define the extra structure for $X$ needed to be $n$-fold category, by insisting that 
the map
\beg{equivmapidentbij}{(S\cup\{n\}-Mor(X))\longleftarrow^{\hspace{-5mm}Id}\hspace{1mm}(S-Mor(X))}
be a bijection.
With this extra structure, $X$ is an $n$-fold category.

\vspace{3mm}

Let $\mathscr{C}$ be an $n$-actad.
Then, by definition,
we have an $n$-fold functor
$$
\diagram
\mathscr{C}\rto&\mathscr{B}_n\rto^{\mathscr{F}_n}& T_{\mathscr{B}_{n-1}}.
\enddiagram
$$
Also, we have 
$$T_{\Xi}:T_X\r T_{\mathscr{B}_{n-1}}.$$
So we can form the pullback $\mathscr{C}\times_{T_{\mathscr{B}_{n-1}}}T_X$.

\begin{definition}

$X$ is called an algebra over $\mathscr{C}$ provided that the following holds:
There exists an $(n-1)$-fold functor 
\beg{ThetaCxTXX}{\Theta:\mathscr{C}\times_{T_{\mathscr{B}_{n-1}}}T_X\r X,}
such that the diagram
$$
\diagram
\mathscr{C}\times_{T_{\mathscr{B}_{n-1}}}T_X\rto^(.65){\Theta}\dto_{\pi}& X\ddto^{\Xi}\\
\mathscr{C}\dto& \\
\mathscr{B}_n\rto_{\mathscr{G}_n}& \mathscr{B}_{n-1},\\
\enddiagram
$$
commutes where $\pi$ is the projection, and the map $\mathscr{C}\r\mathscr{B}_n$ is the fibered $n$-functor
given by definition. Note that though some of the functors in this diagram are $n$-fold functors, this diagram
is of $(n-1)$-fold functors. We require the following conditions to hold:

\vspace{3mm}

\begin{enumerate}
\item Equivariance:

\vspace{3mm}

\noindent 
By the above comment on the $n$-fold categorical structure of $X$,
We have that $\mathscr{C}\times_{T_{\mathscr{B}_{n-1}}}T_X$ is already an $n$-fold category.
Then we require \rref{ThetaCxTXX} to be an $n$-fold functor.

\vspace{5mm}

\item \label{algassoc1}
Associativity:

\vspace{3mm}

\noindent First, fix $x,y\in S-Mor(\mathscr{B}_n)$, $S\subseteq\{1,\dots, n-1\}$, $i\in\{1,\dots,m_x\}$, with
$f_i(x)=\mathscr{G}_n(y)$. Or, in other words, $(x,i,y)\in Comp_n$.
Then write
$$x_j=f_j(x),$$
$$y_k=f_k(y).$$

Suppose we have 
$$\xi\in\mathscr{C}(x)\times \mathscr{C}(y)\times \{(x,i,y)\}.$$
Recall that elements of $\mathscr{C}(x)$ are $S$-morphisms of $\mathscr{B}_n$ over $x$.
Then define
$$\Omega(\xi)=(z_1,\dots,z_{m_x+m_y-1})$$
by 
$$z_{\phi_n^{(x,i,y)}(j)}=x_j,$$
$$z_{\psi_n^{(x,i,y)}(k)}=y_k.$$
Then 
$$\Omega:\mathscr{C}(x)\times\mathscr{C}(y)\times \{(x,i,y)\}\r T_{\mathscr{B}_{n-1}}.$$
We also have 
$$T_\Xi:T_X\r T_{\mathscr{B}_{n-1}}.$$
So, we can write
$$(\mathscr{C}(x)\times \mathscr{C}(y)\times \{(x,i,y)\})\times_{T_{\mathscr{B}_{n-1}}} T_X.$$

Then define 
$$\widetilde{\Theta}_{(x,i,y)}:(\mathscr{C}(x)\times \mathscr{C}(y)\times\{(x,i,y)\})\times_{T_{\mathscr{B}_{n-1}}} T_X\r$$
$$\r\mathscr{C}(x)\times_{T_{\mathscr{B}_{n-1}}}\prod_{1\leq j\leq m_x}X(x_j)$$
by 
\beg{Thetatildesalg}{\begin{array}{c}
\widetilde{\Theta}_{(x,i,y)} = (\varpi,
Id_{\mathscr{C}(x) \times \prod_{1\leq j \leq m_x, \; j\neq i} X(x_j)}\times\\
\Theta\;|_{\mathscr{C}(y) \times_{T_{\mathscr{B}_{n-1}}}\prod_{i\leq k \leq m_y} X(y_k))}.
\end{array}}
where, again,
$$\varpi:(\mathscr{C}(x)\times\mathscr{C}(y)\times\{(x,i,y)\})\times_{T_{\mathscr{B}_{n-1}}} T_X\r\mathscr{C}(x)$$
is the projection to the first coordinate.

Now, let
$$\widetilde{\Theta}=\coprod_{(x,i,y)\in Comp_n}\widetilde{\Theta}_{(x,i,y)}.$$
Then, since
$$Comp_{\mathscr{C}}=(\mathscr{C}\times\mathscr{C})\times_{(\mathscr{B}_n\times\mathscr{B}_n)}Comp_n,$$
we have
$$\widetilde{\Theta}:Comp_{\mathscr{C}}\times_{T_{\mathscr{B}_{n-1}}}T_X\r\mathscr{C}\times_{T_{\mathscr{B}_{n-1}}} T_X.$$

Then
\beg{assoccond4alg}{
\diagram
Comp_{\mathscr{C}}\times_{T_{\mathscr{B}_{n-1}}}T_X\rto^(.55){\Gamma\times Id}\dto_{\widetilde{\Theta}}& \mathscr{C}\times_{T_{\mathscr{B}_{n-1}}}T_X\dto^{\Theta} \\
\mathscr{C}\times_{T_{\mathscr{B}_{n-1}}}T_X\rto_(.55){\Theta}& X\\
\enddiagram
}
must strictly commute.
Actually, $\widetilde{\Theta}$ can be extended uniquely and naturally to $S\cup\{n\}$-morphisms, and \rref{assoccond4alg} is automatically
a diagram of $n$-fold functors.

\end{enumerate}

\end{definition}

Again, at $n=1$, this simply corresponds to the concept of algebras over an operad.

\vspace{5mm}

\noindent {\bf Examples:}
1: 
Let $\mathscr{A}_n$ be the $n$-actad given by taking, for $S\subseteq\{1,\dots,n-1\}$,
$$S-Mor(\mathscr{A}_n)=(S\cup\{n\})-Mor(\mathscr{B}_n),$$
and $(S\cup\{n\})-Mor(\mathscr{A}_n)$ is the $n$-composable pairs of $(S\cup\{n\})-Mor(\mathscr{B}_n)$,
i.e. the pairs 
$$(\sigma,\tau)\in((S\cup\{n\})-Mor(\mathscr{B}_n))^2$$
such that $S_n \tau=T_n \sigma \in S-Mor(\mathscr{B}_n)$, where $S_n$, $T_n$ denote the $n$-source and $n$-target
of an $S\cup \{n\}$-morphism.
Then there is an equivalence of categories between $\mathscr{A}_n$-algebras and $(n-1)$-actads.
More specific examples of $\mathscr{A}_n$-algebras will be given in later sections.

2: Note that, of course, $\mathscr{B}_n$ is an $n$-actad via $Id:\mathscr{B}_n\r\mathscr{B}_n$. We call
$\mathscr{B}_n$-algebras {\em$n$-commutative $(n-1)$-actads}, (in the narrower sense).

\vspace{5mm}

\section{Modules of Algebras over $n$-Actads}\label{smod}

\vspace{5mm}

Suppose $X,M$ are $(n-1)$-fold categories
and we have fibered $(n-1)$-fold functors
$$\Xi_X:X\r\mathscr{B}_{n-1}$$
$$\Xi_M:M\r\mathscr{B}_{n-1}.$$
Then let, again, $X(x)$ and $M(x)$ denote the inverse images under $\Xi_X$ and $\Xi_M$ of $x\in\mathscr{B}_{n-1}$ respectively.
For $S\subseteq\{1,\dots,n-1\}$, define $T_{X,M}$ by
\beg{TXMdefnalgmon}{\begin{array}{c}
S-Mor(T_{X,M})=\{((\iota,k),(x_1,\dots,x_{\iota -1},m, x_{\iota +1},\dots,x_k))\mid\\
1\leq \iota\leq k \in\N,\\
x_j\in S-Mor(X), m\in S-Mor(M)\},\\
\end{array}}

\vspace{3mm}

$$\begin{array}{c}
(S\cup\{n\})-Mor(T_{X,M})=\{((\sigma,\iota,k),(x_1,\dots,x_{\iota-1},m,x_{\iota+1},\dots, x_k))\mid\\
1\leq \iota\leq k\in\N\\
x_j\in S-Mor(X),m\in S-Mor(M), \sigma\in \Sigma_k\}.\\
\end{array}$$

For an $\eta=((\iota,k),(x_1,\dots,x_{\iota-1},m,x_{\iota+1},\dots,x_k))\in S-Mor(T_{X,M})$, write
$\iota_{\eta}=\iota$.

\vspace{5mm}

\begin{definition}

Suppose $\mathscr{C}$ is an $n$-actad and $X$ is a $\mathscr{C}$-algebra.
Then $M$ is a {\em $(\mathscr{C},X)$-module} provided that the following holds:
There exists an $(n-1)$-fold functor
\beg{TXMThetaCxTBM}{\Theta:\mathscr{C}\times_{T_{\mathscr{B}_{n-1}}}T_{X,M}\r M}
such that
$$
\diagram
\mathscr{C}\times_{T_{\mathscr{B}_{n-1}}}T_{X,M}\rto^(.65){\Theta}\dto_{\pi}& M\ddto^{\Xi_M}\\
\mathscr{C}\dto& \\
\mathscr{B}_n\rto_{\mathscr{G}}& \mathscr{B}_{n-1},\\
\enddiagram
$$
where $\pi$ is the projection, and the map $\mathscr{C}\r\mathscr{B}_n$ is the $n$-functor given by the definition.
Note that though some of the functors in this diagram are $n$-fold functors, this diagram is of $(n-1)$-fold functors.
We also require the following conditions to hold:
\begin{enumerate}
\item

Equivariance:

\vspace{3mm}

\noindent Suppose we have $S\subseteq \{1,\dots,n-1\}$.
Then define the extra structure for $M$ needed to be $n$-fold category, by insisting that 
the map
$$(S\cup\{n\}-Mor(M))\longleftarrow^{\hspace{-5mm}Id}\hspace{1mm}(S-Mor(M))$$
be an isomorphism.
With this extra structure, $M$ is an $n$-fold category.
We have that $\mathscr{C}\times_{T_{\mathscr{B}_{n-1}}}T_{X,M}$ is already an $n$-fold category.
Then \rref{TXMThetaCxTBM}
must be an $n$-fold functor.

\vspace{5mm}

\item 
Associativity:

\vspace{3mm}

\noindent We will use the symbols $x,y,S,i,x_j,y_k,\Omega$ in the same way as in Axiom $($\ref{algassoc1}$)$ in Section \ref{salg}.
We have $$T_{\Xi_X}:T_X\r T_{\mathscr{B}_{n-1}},$$
We define
$$\Upsilon:T_{X,M}\r T_{\mathscr{B}_{n-1}}$$
by
$$\Upsilon(x_1,\dots, x_{i-1}, m, x_{i+1},\dots, x_k)=$$
$$=(\Xi_X(x_1),\dots,\Xi_X(x_{i-1}), \Xi_M(m),\Xi_X(x_{i+1}),\dots,\Xi_X(x_k)).$$
So, we can write
$$(\mathscr{C}(x)\times\mathscr{C}(y)\times\{(x,i,y)\})\times_{T_{\mathscr{B}_{n-1}}} T_{X,M}.$$
Then define 
$$\widetilde{\Theta}_{(x,i,y)}:(\mathscr{C}(x)\times\mathscr{C}(y)\times \{(x,i,y)\})\times_{T_{\mathscr{B}_{n-1}}} T_{X,M}\r$$
$$\r \mathscr{C}(x)\times_{T_{\mathscr{B}_{n-1}}}\prod_{1\leq j\leq m_x} M(x_j)$$
by 
$$\begin{array}{c}
\widetilde{\Theta}_{(x,i,y)} = (\varpi,
Id_{\mathscr{C}(x) \times \prod_{1\leq j \leq m_x, \; j\neq i} XM_j(x_j)}\times\\
\Theta\;|_{\mathscr{C}(y) \times_{T_{\mathscr{B}_{n-1}}}\prod_{i\leq k \leq m_y} XM_k(y_k))}.
\end{array}
$$
where, for $\eta\in T_{X,M}$ on $\mathscr{C}(x)\times\mathscr{C}(y)\times\{(x,i,y)\}\times_{T_{\mathscr{B}_{n-1}}}\{\eta\}$,
$$\begin{array}{c}
XM_j=X \text{ for } j\neq \iota\\
\hspace{10mm} =M \text{ for } j=\iota .\\
\end{array}$$

Then
\beg{assoccond4modu}{
\diagram
Comp_{\mathscr{C}}\times_{T_{\mathscr{B}_{n-1}}}T_{X,M}\rto^(.55){\Gamma\times Id}\dto_{\widetilde{\Theta}}& \mathscr{C}\times_{T_{\mathscr{B}_{n-1}}}T_{X,M}\dto^{\Theta} \\
\mathscr{C}\times_{T_{\mathscr{B}_{n-1}}}T_{X,M}\rto_(.55){\Theta}& M\\
\enddiagram
}
must strictly commute.
\end{enumerate}
\end{definition}

We can also define a concept of a $(\mathscr{C}, X)$-algebra $M$ analogously by modifying the definition so as to allow
multiple $m_j$'s in \rref{TXMdefnalgmon}.

\vspace{5mm}

\section{Iterated Algebras}\label{siter}

\vspace{5mm}

There is a concept of an {\em $I$-iterated algebra} over an $n$-actad for any subset $I\in\{0,\dots,n\}$, of which the concepts of algebra and
algebra with a module introduced in the last two sections are special cases for $I=\{n\}$ and $I=\{0,n\}$, respectively.
We will use these constructions later to define unital actads and $R$-units.
An $I$-iterated algebra consists of $|I|$ cube-like $(n-1)$-fold categories, with rules on where they can be ``plugged in into an element of a given
$n$-actad." The ways of plugging in the $\alpha$'th model, for an $\alpha\in I$, follow, roughly, the pattern of an $\alpha$-actad algebra.
However, the precise axiomatization is delicate, in particular, on morphisms.
The only axiomatization I could work out uses multisorted algebras \cite{multi}. One must
also design an appropriate multisorted version of the $n$-base, to keep track of the structure describing how the different models are being
plugged in.

In the below picture, I illustrate how the $0$-, $1$-, and $2$-models are plugged in in the composition for the $0$-, $1$-, and $2$-models
of a $\{0,1,2\}$-iterated algebra over a $2$-actad.

\vspace{10mm}

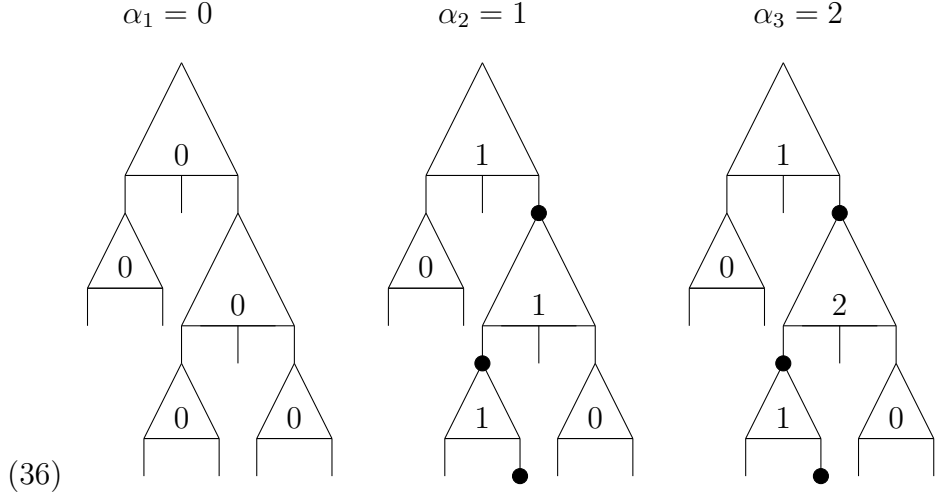
$$\alpha_1=0 \hspace{30mm}
\alpha_2=1 \hspace{30mm}
\alpha_3=2$$
\beg{IteratedExamplePicture}{\begin{tikzpicture}
\draw(0,1.5)--(0,2);
\draw(0,2)--(0.5,3);
\draw(0.5,3)--(1,2);
\draw(1,1.5)--(1,2);
\draw(0,2)--(1,2);
\draw(0.5,3)--(0.5,3.5);
\draw(2,3.5)--(2,3);
\draw(1.25,3.5)--(1.25,3);
\draw(0.5,3.5)--(2,3.5);
\draw(0.5,3.5)--(1.25,5);
\draw(1.25,5)--(2,3.5);
\draw(2,3)--(1.25,1.5);
\draw(2.75,1.5)--(2,3);
\draw(2.75,1.5)--(1.25,1.5);
\draw(2.75,1.5)--(2.75,1);
\draw(1.25,1.5)--(1.25,1);
\draw(2,1.5)--(2,1);
\draw(1.25,1)--(0.75,0);
\draw(1.25,1)--(1.75,0);
\draw(0.75,0)--(1.75,0);
\draw(0.75,0)--(0.75,-0.5);
\draw(1.75,0)--(1.75,-0.5);
\draw(2.75,1)--(2.25,0);
\draw(2.75,1)--(3.25,0);
\draw(2.25,0)--(3.25,0);
\draw(3.25,0)--(3.25,-0.5);
\draw(2.25,0)--(2.25,-0.5);
\draw (0,2) -- node[above] {0} ++(1,0);
\draw (0.5,3.5) -- node[above] {0} ++(1.5,0);
\draw (1.5,1.5) -- node[above] {0} ++(1,0);
\draw (0.75,0) -- node[above] {0} ++(1,0);
\draw (2.25,0) -- node[above] {0} ++(1,0);

\draw(4,1.5)--(4,2);
\draw(4,2)--(4.5,3);
\draw(4.5,3)--(5,2);
\draw(5,1.5)--(5,2);
\draw(4,2)--(5,2);
\draw(4.5,3)--(4.5,3.5);
\draw(6,3.5)--(6,3);
\draw(5.25,3.5)--(5.25,3);
\draw(4.5,3.5)--(6,3.5);
\draw(4.5,3.5)--(5.25,5);
\draw(5.25,5)--(6,3.5);
\draw(6,3)--(5.25,1.5);
\draw(6.75,1.5)--(6,3);
\draw(6.75,1.5)--(5.25,1.5);
\draw(6.75,1.5)--(6.75,1);
\draw(5.25,1.5)--(5.25,1);
\draw(6,1.5)--(6,1);
\draw(5.25,1)--(4.75,0);
\draw(5.25,1)--(5.75,0);
\draw(4.75,0)--(5.75,0);
\draw(4.75,0)--(4.75,-0.5);
\draw(5.75,0)--(5.75,-0.5);
\draw(6.75,1)--(6.25,0);
\draw(6.75,1)--(7.25,0);
\draw(6.25,0)--(7.25,0);
\draw(7.25,0)--(7.25,-0.5);
\draw(6.25,0)--(6.25,-0.5);
\draw (4,2) -- node[above] {0} ++(1,0);
\draw (4.5,3.5) -- node[above] {1} ++(1.5,0);
\draw (5.5,1.5) -- node[above] {1} ++(1,0);
\draw (4.75,0) -- node[above] {1} ++(1,0);
\draw (6.25,0) -- node[above] {0} ++(1,0);

\draw [fill] (6,3) circle [radius=0.1];
\draw [fill] (5.25,1) circle [radius=0.1];
\draw [fill] (5.75,-0.5) circle [radius=0.1];

\draw(8,1.5)--(8,2);
\draw(8,2)--(8.5,3);
\draw(8.5,3)--(9,2);
\draw(9,1.5)--(9,2);
\draw(8,2)--(9,2);
\draw(8.5,3)--(8.5,3.5);
\draw(10,3.5)--(10,3);
\draw(9.25,3.5)--(9.25,3);
\draw(8.5,3.5)--(10,3.5);
\draw(8.5,3.5)--(9.25,5);
\draw(9.25,5)--(10,3.5);
\draw(10,3)--(9.25,1.5);
\draw(10.75,1.5)--(10,3);
\draw(10.75,1.5)--(9.25,1.5);
\draw(10.75,1.5)--(10.75,1);
\draw(9.25,1.5)--(9.25,1);
\draw(10,1.5)--(10,1);
\draw(9.25,1)--(8.75,0);
\draw(9.25,1)--(9.75,0);
\draw(8.75,0)--(9.75,0);
\draw(8.75,0)--(8.75,-0.5);
\draw(9.75,0)--(9.75,-0.5);
\draw(10.75,1)--(10.25,0);
\draw(10.75,1)--(11.25,0);
\draw(10.25,0)--(11.25,0);
\draw(11.25,0)--(11.25,-0.5);
\draw(10.25,0)--(10.25,-0.5);
\draw (8,2) -- node[above] {0} ++(1,0);
\draw (8.5,3.5) -- node[above] {1} ++(1.5,0);
\draw (9.5,1.5) -- node[above] {2} ++(1,0);
\draw (8.75,0) -- node[above] {1} ++(1,0);
\draw (10.25,0) -- node[above] {0} ++(1,0);

\draw [fill] (10,3) circle [radius=0.1];
\draw [fill] (9.25,1) circle [radius=0.1];
\draw [fill] (9.75,-0.5) circle [radius=0.1];

\end{tikzpicture}}
\vspace{5mm}

\noindent(the solid nodes mark the places where trees that will be labeled a number $\geq 0$ are composed).

\vspace{10mm}

We shall start with the concept of the {\em $I$-pointed $n$-base} $\mathscr{B}_n^I$, a cube-like $n$-fold category, which will be defined inductively.

We start the induction with $\mathscr{B}_0^{\{0 \}} = \mathscr{B}_0^\emptyset = \mathscr{B}_0$.
For $k\geq 0$, $0\leq \alpha_1<\dots <\alpha_k\leq n$, $I=\{\alpha_1,\dots,\alpha_k\}$,
$\mathscr{B}_n^I=\mathscr{B}_n^{\alpha_1,\dots,\alpha_k}$ is an $n$-fold category with a $n$-fold functor
$$\mathscr{B}_n^{\alpha_1, \dots , \alpha_k} \rightarrow \mathscr{B}_n$$
and with some additional maps.
First, $\mathscr{B}_n^{\alpha_1 ,\dots ,\alpha_k}$ comes with additional data of $n$-fold functor
$$\wp_n:\mathscr{B}_n^{\alpha_1,\dots,\alpha_k}\r\{1,\dots,k\}$$
as part of the induction data, where on the right hand side, all morphisms are identities.
We shall also write
$$\mathscr{B}_n^I(i)=\mathscr{B}_n^{\alpha_1,\dots,\alpha_k} (i)=\wp_n^{-1}(i).$$

We will construct $\mathscr{B}_n^{\alpha_1,\dots,\alpha_k}$ (and the associated structure)
inductively on $n$. The induction step from $(n-1)$
to $n$ depends in a crucial way on whether $\alpha_k<n$ or $\alpha_k=n$.
Essentially, if $\alpha_k<n$, the induction step from $\mathscr{B}_{n-1}^{\alpha_1,\dots,\alpha_k}$ is similar as for $\mathscr{B}_n$. If $\alpha_k=n$,
the induction step to $\mathscr{B}_n^{\alpha_1,\dots , \alpha_k}$
will be from $\mathscr{B}_{n-1}^{\alpha_1,\dots,\alpha_{k-1}}$ to $\mathscr{B}_n^{\alpha_1,\dots,\alpha_k}$. 
The rules for plugging in the new model are essentially to plug into one of the places of the previous one, but the
$(S\cup\{n\})$-morphisms must remember their positions.

\vspace{3mm}

In more detail, if $\alpha_k< n$, then
we have an $(n-1)$-fold functor over $\{1,\dots, k\}$
$$
\diagram
\mathscr{F}_n^{\alpha_1,\dots,\alpha_k}:\mathscr{B}_n^{\alpha_1,\dots,\alpha_k}\rto\drto& T_{\mathscr{B}_{n-1}^{\alpha_1,\dots,\alpha_k}}\dto\\
 &\{1,\dots,k\},\\
\enddiagram$$
where in the vertical arrow, an $S$-morphism $(x_1,\dots,x_m)$ with $S\subseteq\{1,\dots, n-1\}$ is mapped to
\beg{maxwpxi}{\max_{i\in\{1,\dots,m\}} \wp_{n-1}(x_i).}
Then, like in the case of $n$-bases, $\mathscr{F}_n^{\alpha_1,\dots,\alpha_k}$ induces an $n$-fold category structure on  $\mathscr{B}_n^{\alpha_1,\dots,\alpha_k}$, and thus becomes an $n$-fold functor.
Again, we also will have an $(n-1)$-fold functor
$$\mathscr{G}_n:\mathscr{B}_n^{\alpha_1,\dots,\alpha_k}\r \mathscr{B}_{n-1}^{\alpha_1,\dots,\alpha_k}$$
as part of the induction data.

\vspace{3mm}

Definitions of $Comp_n^{\alpha_1,\dots,\alpha_k}$ and $Comp_n^{2;\alpha_1,\dots,\alpha_k}$ then proceed
precisely analogously to the definitions of $Comp_n$ and $Comp_n^2$ in Section \ref{s2}, except for the fact that we are working
over $\{1,\dots,k\}$. The analogue of Diagram \rref{compdiagram} becomes a diagram of the form
\beg{Compdiagramiter}{
\diagram
Comp_n^{2;\alpha_1,\dots,\alpha_k}\dto_{\Phi_n^{2;\alpha_1,\dots,\alpha_k}}\rto^(0.6){\Phi_n^{1;\alpha_1,\dots,\alpha_k}} 
&Comp_n^{\alpha_1,\dots,\alpha_k}\dto\\
Comp_n^{\alpha_1,\dots,\alpha_k}\rto& \mathscr{B}_n^{\alpha_1,\dots,\alpha_k}\\
\enddiagram}
over $\{1,\dots,k\}$.

\vspace{3mm}

Now, suppose $\alpha_k=n$.
This case must be treated in more detail. Define a map
$$\varpi:\{1,\dots,k\}\r\{1,\dots,k-1\}$$
by
$$\varpi(i)=i, \text{ for}\; i<k$$
$$\varpi(k)=k-1.$$

\vspace{3mm}

\begin{definition}\label{DefinitionPointedTMultisorted}
Suppose $\mathscr{C}\r\{1,\dots,k-1\}$ is a $(n-1)$-fold category over $\{1,\dots,k-1\}$.
Then define an $n$-fold category $T^{\bullet}_{\mathscr{C}}\r\{1,\dots,k\}$ over $\{1,\dots,k\}$ as follows:
For $i\in\{1,\dots, k-1\}$,
$$T_{\mathscr{C}}^\bullet(i)=T_{\mathscr{C}(i)}.$$
For $S\subseteq \{1,\dots,n-1\}$,
$$S-Mor(T^{\bullet}_{\mathscr{C}}(k))=$$
$$\{((x_1,\dots,x_{\ell}),b)|\ell\geq 1, \; b\in\{1,\dots,\ell\}, \; \forall j\; \exists i_j \in \{1, \dots, k-1\}$$
$$\text{such that } x_j\in S-Mor(\mathscr{C}(i_j)), \; i_b=k-1,\; \}$$

\vspace{3mm}

$$(S\cup\{n\})-Mor(T^{\bullet}_{\mathscr{C}}(k))=$$
$$\{f=(((x_1,\dots,x_{\ell}),b),\sigma, (y_1,\dots,y_{\ell}))$$
$$:((x_1,\dots,x_{\ell}),b)\r((y_1,\dots,y_{\ell}),\sigma(b))| \ell\geq 1, $$
$$b\in\{1,\dots,\ell\}, \sigma:\{1,\dots,\ell\}\r\{1,\dots,\ell\} \text{ a permutation},$$
$$\sigma(b)=b, \forall i \; y_i=x_{\sigma(i)}\}.$$
\end{definition}

We have an obvious forgetful $n$-fold functor
$$\diagram
T_{\mathscr{C}}^\bullet\dto\rto& T_\mathscr{C}\dto\\
\{1,\dots,k\}\rto^\varpi&\{1,\dots, k-1\}\\
\enddiagram$$
forgetting $b$.

\vspace{3mm}

Then we are also given a $(n-1)$-fold functor
\beg{FNIInductionDataaddingtoset}{\mathscr{F}_n^I:\mathscr{B}_n^I = \mathscr{B}_n^{\alpha_1,\dots , \alpha_k}
\r T^{\bullet}_{\mathscr{B}_{n-1}^{\alpha_1,\dots,\alpha_{k-1}}}}
as part of the induction data.
Like in the previous cases,
$\mathscr{F}_n^I$ becomes an $n$-fold functor.
In addition, like in the previous cases,
we have a $(n-1)$-fold functor
$$\mathscr{G}_n^I:\mathscr{B}_n^I\r\mathscr{B}^{\alpha_1,\dots,\alpha_{k-1}}_{n-1},$$
with the following diagram
$$
\diagram
\mathscr{B}_{n}^I\rto^{\mathscr{G}_n}\dto&\mathscr{B}_{n-1}^{\alpha_1,\dots,\alpha_{k-1}}\dto\\
\{1,\dots,k\}\rto^{\varpi}&\{1,\dots,k-1\}.\\
\enddiagram
$$

\vspace{3mm}

Then for $S\subseteq \{1,\dots, n-1\}$, $x\in S-Mor(\mathscr{B}_n^I)$, denote
$$\mathscr{F}_n^I(x)=:(f_1(x),\dots,f_{m_x}(x)).$$

Now, we will use these functors to define the $n$-fold categories $Comp_n^{\alpha_1,\dots,\alpha_k}$, $Comp_n^{2,\alpha_1,\dots,\alpha_k}$
similarly to the case of $n$-bases:
For $S\subseteq \{1,\dots,n-1\}$, define

$$S-Mor(Comp_n^{\alpha_1,\dots,\alpha_k})=$$
$$\{(x,i,y)| x, y\in S-Mor(\mathscr{B}_n^{\alpha_1,\dots,\alpha_k}), 1\leq i\leq m_x, \mathscr{G}_n(y)=f_i(x)\}$$

\vspace{1.5mm}

$$(S\cup\{n\})-Mor(Comp_n^{\alpha_1,\dots,\alpha_k})=\{((x,i,y)\r(x,i',y), f, g)|1\leq i\leq m_x,$$
$$x,y\in S-Mor(\mathscr{B}_n^{\alpha_1,\dots,\alpha_k}), \mathscr{G}_n(y)=f_i(x),$$
$$f:x\r x,\; g:y\r y,\; f,g\in (S\cup\{n\})-Mor(\mathscr{B}_n), i'=\sigma_f(i)\}.$$

\vspace{3mm}

$$S-Mor(Comp^{2,\alpha_1,\dots,\alpha_k}_n)=\{(x,i,j,y,z)|x,y,z\in S-Mor(\mathscr{B}_n^{\alpha_1,\dots,\alpha_k}),$$
$$1\leq i\neq j\leq m_x, \mathscr{G}_n(y)=f_i(x), \mathscr{G}_n (z)=f_j(x)\}.$$

\vspace{1.5mm}

$$(S\cup\{n\})-Mor(Comp_n^{2,\alpha_1,\dots,\alpha_k})=\{((x,i,j,y,z)\r(x,i',j',y,z),f,g,h)|$$
$$1\leq i\neq j\leq m_x,$$
$$x,y,z\in S-Mor(\mathscr{B}_n^{\alpha_1,\dots,\alpha_k}), \mathscr{G}_n(y)=f_i(x), \mathscr{G}_n(z)=f_j(x),$$
$$f:x\r x,\; g:y\r y,\; h:z\r z,\; f,g,h\in S\cup\{n\}-Mor(\mathscr{B}_n),$$
$$i'=\sigma_f(i), j'=\sigma_f(j)\}.$$

Then define maps
$$\Pi_1:Comp_n^{\alpha_1,\dots,\alpha_k}\r \{1,\dots,k\}$$
$$\Pi_2:Comp^{2,\alpha_1,\dots,\alpha_k}\r \{1,\dots,k\}$$
by $$\Pi_1(x,i,y)=\varpi(x)$$
$$\Pi_2(x,i,j,y,z)=\varpi(x).$$

Also, like before, we have
$$Comp_n^{\alpha_1,\dots,\alpha_k}\r \mathscr{B}_n^{\alpha_1,\dots,\alpha_k}\times \mathscr{B}_n^{\alpha_1,\dots,\alpha_k}$$
$$Comp_n^{2,\alpha_1,\dots,\alpha_k}\r \mathscr{B}_n^{\alpha_1,\dots,\alpha_k}\times\mathscr{B}_n^{\alpha_1,\dots,\alpha_k}\times
\mathscr{B}_n^{\alpha_1,\dots,\alpha_k}$$
$$\varphi_n^{(x,i,y)}:\{1,\dots,m_x\}\r \{1,\dots, m_x+m_y-1\}$$
$$\Phi_n:Comp_n^{\alpha_1,\dots,\alpha_k}\r \mathscr{B}_n^{\alpha_1, \dots , \alpha_k},$$
and
$$\Phi_n^1,\Phi_n^2:Comp_n^{2,\alpha_1,\dots,\alpha_k}\r Comp_n^{\alpha_1,\dots,\alpha_k},$$
where the diagram analogous to \rref{Compdiagramiter} commutes.

\vspace{5mm}

Then suppose we are given the $\mathscr{B}_n^{\alpha_1,\dots,\alpha_k}$ and all of their associated structure described above.

To construct, for $\{\alpha_1,\dots,\alpha_k\}\subseteq\{0,\dots,n\}$, $\mathscr{B}_{n+1}^{\alpha_1,\dots,\alpha_k}$
from $\mathscr{B}_n^{\alpha_1,\dots,\alpha_k}$, we can proceed precisely analogously to the end of Section \ref{s2}:
For $S\subseteq\{1,\dots,n\}$, $S-Mor(\mathscr{B}_{n+1}^{\alpha_1,\dots,\alpha_k})$ is
the $(n+1)$-dimensional free plain actad sorted over $\{1,\dots,k\}$
(i.e. plain actad in the category of sets over $\{1,\dots,k\}$) on
$$\mathscr{S}_{\underline{0}}:S-Mor(\mathscr{B}_n^{\alpha_1,\dots,\alpha_k})\r B_n\times\{1,\dots,k\}.$$
Again, this is an invertible cube-like category and $\underline{0}$ could be equivalently replaced
by any $\mathbb{I}:S\r\{0,1\}$.
We let 
$$\mathscr{B}_{n+1}^{\alpha_1,\dots,\alpha_k,n+1}=\mathscr{B}_{n+1}^{\alpha_1,\dots,\alpha_k}\times_{T_{\mathscr{B}_n^{\alpha_1,\dots,\alpha_k}}} T^{\bullet}_{\mathscr{B}_n^{\alpha_1,\dots,\alpha_k}}.$$
(In particular, by definition, $\mathscr{B}_n^0 = \mathscr{B}_n$ for all $n$).

\vspace{5mm}

Now, we will define iterated algebras.
Let $\mathscr{C}$ be an $n$-actad. Let $0\leq\alpha_1<\dots<\alpha_k\leq n$. Let
$$\mathscr{C}^{\alpha_1,\dots,\alpha_k}=\mathscr{C}\times_{\mathscr{B}_n}\mathscr{B}_n^{\alpha_1,\dots,\alpha_k}.$$
$$Comp_{\mathscr{C}}^{\alpha_1,\dots,\alpha_k}=(\mathscr{C}\times\mathscr{C})\times_{(\mathscr{B}_n\times\mathscr{B}_n)}Comp_n^{\alpha_1,\dots,\alpha_k}$$
$$Comp_{\mathscr{C}}^{2;\alpha_1,\dots,\alpha_k}=(\mathscr{C}\times\mathscr{C}\times\mathscr{C})\times_{(\mathscr{B}_n\times\mathscr{B}_n\times\mathscr{B}_n)}Comp_n^{2;\alpha_1,\dots,\alpha_k}.$$
(Note that we have a forgetful $n$-fold functor $\mathscr{B}_n^{\alpha_1,\dots,\alpha_k}\r\mathscr{B}_n$.)

\begin{definition}
Suppose $\mathscr{D}$ is an $(n-1)$-fold category  over $\mathscr{B}_{n-1}\times\{1,\dots,k\}$. Then, for $i=1,\dots k$, let
$\mathscr{D}(i)$ be the sub-$(n-1)$-fold categories that are fibered over $\mathscr{B}_{n-1}\times \{i\}$.
Then let $T_{\mathscr{D}}^k$ be the $n$-fold category with
$$S-Mor(T_{\mathscr{D}}^k)=\{(x_1,\dots,x_k)|\forall j\; \exists i_j \in \{1,\dots , k-1\}$$ 
$$\text{such that } x_j\in S-Mor(\mathscr{D}(i_j)\}$$

\vspace{3mm}

$$S\cup\{n\}-Mor(T_{\mathscr{D}}^k)=\{f=((x_1,\dots,x_k),\sigma,(y_1,\dots,y_k))$$
$$:(x_1,\dots,x_k)\r(y_1,\dots,y_{k})|$$
$$\sigma:\{1,\dots,k\}\r\{1,\dots,k\} \text{ a permutation,}$$
$$\forall j\;\; x_j\in S-Mor(\mathscr{D}(i_j)), \; y_j\in S-Mor(\mathscr{D}(\sigma (i_j))), \text{ and } y_j=x_{\sigma(i)}\}$$
for $S\subseteq \{1,\dots,k-1\}$.
\end{definition}

An iterated $(\alpha_1,\dots,\alpha_k)-\mathscr{C}$ algebra is a $k$-tuple of fibered $(n-1)$-fold categories
$X=(X_1,\dots,X_k)$ over $\mathscr{B}_{n-1}$ together with an $(n-1)$-fold functor
$$\Theta :\mathscr{C}^{\alpha_1,\dots,\alpha_k}\times_{T_{\mathscr{B}_{n-1}\times\{1,\dots,k\}}^k}T_X^k\r X$$
with
$$\diagram
\mathscr{C}^{\alpha_1,\dots,\alpha_k}\times_{T_{\mathscr{B}_{n-1}\times\{1,\dots,k\}}^k}T_X^k\rto^(.725){\Theta}\drto& X\dto\\
 &\{1,\dots,k\}.\\
\enddiagram
$$
Again, there are equivariance and associativity axioms. Equivariance, again, says that $\Theta$ is an $n$-fold functor if we let \rref{equivmapidentbij}
be a bijection. Associativity is expressed by the commutativity of a diagram of $n$-fold:
$$
\diagram
Comp_{\mathscr{C}}^{\alpha_1,\dots,\alpha_k}\times_{T_{\mathscr{B}_{n-1}\times\{1,\dots,k\}}^k}T_X^k\rto^{\Gamma\times Id}\dto_{\widetilde{\Theta}}& \mathscr{C}^{\alpha_1,\dots,\alpha_k}\times_{T_{\mathscr{B}_{n-1}\times\{1,\dots,k\}}^k}T_X^k\dto^{\Theta}\\
\mathscr{C}^{\alpha_1,\dots,\alpha_k}\times_{T_{\mathscr{B}_{n-1}\times\{1,\dots, k\}}^k} T_X^k\rto^{\Theta}& X\\
\enddiagram
$$
where $\widetilde{\Theta}$ is defined analogously to \rref{Thetatildesalg}.

\vspace{5mm}


As an example, note that the cases of $\mathscr{B}_n^{\{n\}}$ and
$\mathscr{B}_n^{\{0,n\}}$ give the notions of algebras and modules (respectively) over $n$-actads.

\vspace{5mm}

\section{$R$-Unital Actads and Bases}\label{sunital}

\vspace{5mm}

A unit axiom can be added to the concept of an $n$-actad (or a plain $n$-actad) without difficulty. Perhaps
the most surprising aspect of units is how diverse they become for $n>1$:
In the plain case, the unit is expressed by the natural inclusion
\beg{inclusionplainbasesruni}{B_{n-1}\subseteq B_n, \;\; n\geq 1.}
(given by the fact that $B_n$ is the free plain $(n-1)$-actad on 
$Id:B_{n-1}\r B_{n-1}$. 

We introduce the notation
$1^n_y\in B_n$ for the image of every $y\in B_{n-1}$.
Then $\forall x=((x_1,\dots,x_\ell), (i_1,\dots,i_{\ell-1}))$
\begin{enumerate}
\item
$x\circ_k^n 1^n_{x_k}=x$
\item
$\forall y\in B_n \text{ such that } G_n(y)=x,$
$$1_{x}^n\circ_1^n y=y.$$
\end{enumerate}

Then, we can define a unital plain actad as follows:
\begin{definition}
Suppose we are given a plain $n$-actad $\mathcal{C}$.
Then $\mathcal{C}$ is a unital plain actad provided that for $x\in B_{n-1}$, $y \in B_n$, $z\in \mathcal{C}(y)$,
we are given an element $1_{n,x} \in \mathcal{C}(1_x^n)$ such that
\begin{enumerate}
\item
$$\gamma_{n,j}(z,1_{n,x})=z$$

\item
$$\gamma_{n,j}(1_{n,x},z)=z,$$
if $G_n(y)=x$.
\end{enumerate}
For $n=0$, there is also a unit $1_{0,*}$, even though there is no $-1$ base.
\end{definition}

A unital $0$-actad is a monoid (instead of just a semigroup).

For actads, the story is analogous. Here, only the case of $n>0$ is new.
By the definition of $\mathscr{B}_{n+1}$ from $\mathscr{B}_n$ at the end of Section \ref{s2}, we also obtain a canonical inclusion $(n-1)$-fold
functor of $(n-1)$-fold categories
\beg{InclusionscrBruni}{\mathscr{I}_n:\mathscr{B}_{n-1}\r\mathscr{B}_n.}
In fact, it becomes an $n$-fold functor if we make all the $\{n\}$-morphisms identities in the source. In fact, \rref{InclusionscrBruni} is a section of the
$(n-1)$-fold functor $\mathscr{G}_n$. (In particular, in the $n$-actad case, note that units have $\{i\}$-automorphisms for $i<n$.)

The left unit property for $\mathscr{B}_n$ then can be expressed by introducing a left unit functor
$$\mathscr{H}_n:\mathscr{B}_n\r Comp_n$$
by
$$x\mapsto (\mathscr{I}_n\mathscr{G}_n x,1,x)$$
(see \rref{CompnS-morsdefn}).
Then the diagram
\beg{BnCompdiaguni}{
\diagram
\mathscr{B}_n\rto^{\mathscr{H}_n}\drto_{Id}& Comp_n\dto^{\Phi_n}\\
 &\mathscr{B}_n\\
\enddiagram}
commutes, and we call it the left unit property.
For the right unit axiom, we must remember where the unit is being inserted.
Recalling the notation of the last section, we may define a right unit $n$-fold functor
$$\mathscr{K}_n:\mathscr{B}_n^0\r Comp_n$$
by
$$\mathscr{K}_n (x,b)=(x,b,\mathscr{I}_n f_b(x)).$$
Then the right unit axiom is the commutativity of the diagram
\beg{Bn0CompDiagUni}{
\diagram
\mathscr{B}_n^0\rto^{\mathscr{K}_n}\drto& Comp_n\dto^{\Phi_n}\\
 &\mathscr{B}_n\\
\enddiagram}
where the diagonal $n$-fold functor is the projection.

\begin{definition}
For an $n$-actad $\mathscr{C}$, let
$$\mathscr{C}^0=\mathscr{C}\times_{\mathscr{B}_n}\mathscr{B}_n^0.$$
Then $\mathscr{C}$ is called a unital $n$-actad if there exists an $n$-fold functor lift
$$\diagram
& \mathscr{C}\dto\\
\mathscr{B}_{n-1}\urto^{\mathscr{I}_n^{\mathscr{C}}}\rto_{\mathscr{I}_n}&\mathscr{B}_n\\
\enddiagram$$
such that, letting
$$\mathscr{H}_n^{\mathscr{C}}:\mathscr{C}\r Comp_n^{\mathscr{C}}$$
be defined by
$$x\mapsto (\mathscr{I}_n^{\mathscr{C}}\mathscr{G}_n |x|,1,x),$$
where $|x|$ denotes the projection of $x$ to $\mathscr{B}_n$,
and
$$\mathscr{K}_n^{\mathscr{C}}:\mathscr{C}^0\r Comp_n^{\mathscr{C}}$$
be defined by
$$\mathscr{K}_n^{\mathscr{C}}=(x,b,\mathscr{I}_n^{\mathscr{C}}|f_b(x)|),$$
the following diagrams above \rref{BnCompdiaguni} and \rref{Bn0CompDiagUni} commute:
$$\diagram
\mathscr{C}\rto^{\mathscr{H}_n^{\mathscr{C}}}\drto_{Id}& Comp_n\dto^{\Gamma}\\
 &\mathscr{C}\\
\enddiagram$$
$$\diagram
\mathscr{C}^0\rto^{\mathscr{K}_n^{\mathscr{C}}}\drto_{\pi}& Comp_n\dto^{\Gamma}\\
 &\mathscr{C}\\
\enddiagram$$
where $\pi$ is the projection.
\end{definition}

But now there is also a version of unitality of $n$-actads (or plain $n$-actad) which includes units coming from $i$-actads
(or plain $i$-actads) for all $i<n$.
I call this concept {\em R-unitality} where R stands for ``recursive."

To give an example, a based operad (which, recall, is over $\mathbb{N}_0$ instead of $\mathbb{N}$)
is a unital operad $\mathscr{C}$ with a base point $*\in\mathscr{C}(0)$.
In a way, this comes from the monoid unit, since the free monoid on one generator contains $1=a^0$,
corresponding to the $0$ index. It is interesting to note that the base point of an operad is not subject to any axioms. Yet, it is an important
feature, adding a rich structure of ``degeneracies" given by plugging in the base point.
This structure is crucial in the Approximation Theorem in infinite loop space theory \cite{MayIterated}.

To define R-unital plain actads, we will define, simultaneously, the {\em R-unital plain $n$-base}
$B_n^R\supset B_n$ with maps
$$F_n^R: B_n^R\r T_{B_{n-1}^R}^R$$
$$G_n^R: B_n^R\r B_{n-1}^R$$
and composition $\circ^n$.
Here $T_X^R$ is defined the same way as $T_X$ (recall Definition \ref{defnTCinduced}), except that an empty sequence is allowed.
In fact, there will be a distinguished subset
$$B_n^{R,0}\subseteq B_n^R$$
such that
$$x\in B_n^{R,0}\Rightarrow m_x=0.$$
Now an {\em $R$-unital plain actad} $\mathscr{C}$ is defined the same as a unital plain $n$-actad with $B_n$ replaced
by $B_n^R$, together with a section
$$B_n^{R,0}\r\mathscr{C}$$
over $B_{n-1}^R$ (using $G_n^R$).

Then $B_{n+1}^R$ is defined as the free plain $R$-unital $n$-actad over $Id:B_n^R\r B_n^R$, modulo the relations
\beg{unitaxiomish}{\gamma_{n,i}(x,y)=x\circ_i^n y, \text{ for } x\in B_n^R, \; y\in B_n^{R,0}.}
One defines, inductively,
\beg{Bn+1R0InductiveStep}{B_{n+1}^{R,0}=B_n^{R,0}\coprod\{1_x^n|x\in B_{n-1}^R\}}
(we want to begin the induction with $B_0^R= \{ *,*\}$, $B_1^R =\mathbb{N}_0$).
Note that $m_x$ in $B_{n+1}$ does not necessarily equal $m_x$ in $B_n$.
We emphasize that in the second term on the right hand side of \rref{Bn+1R0InductiveStep} we mean the $R$-unital $n$-actad units,
{\em not} the units $1^n_x\in B_n^R$.

The reason to keep these units separate is to be able to plug into $1^n_x$,  in $\circ^{n+1}$: we have $1^n_x\in B_{n+1}^R$
via the inclusion $B_n^R\subset B_{n+1}^R$ with $m_{1^n_x}=1$. To explain this, in our definition of trees
in $B_2^R$, there is no reason to prohibit plugging in a tree with $1$ prong into the triangle with one prong 
(``lozenge") $1^1_1$ via $\circ^2$. However, plugging in the R-unit $1_1^1$ via $\circ^2$ eliminates the lozenge from the
tree, which is a new operation. 

The structure of $B_n^R$ is not as complicated as it may seem. In fact, one readily proves by induction
\begin{lemma}
The canonical inclusion $B_n\subseteq B_n^R\setminus B_n^{R,0}$ is a bijection.
\end{lemma}
\qed

Now $\mathscr{B}_n^R$, {\em $R$-unital $n$-actads}, (and all associated functors)
can be defined completely analogously to, $\mathscr{B}_n$, $n$-actads, via an inductive definition of $\mathscr{B}_n^R$ (and 
the associated functors) and then replacing $\mathscr{B}^R_n$ while extending the formula \rref{maxwpxi} to $S$-morphisms, with
$S\subseteq\{1,\dots,n\}$.
This is then a generalization of (unital) $n$-actads.

In this approach, we keep permutations as the $\{n\}$-morphisms; it is possible to create a larger
category generated by these morphisms and $n$-compositions with the various $i$-units, $i<n$, but then we no longer have an invertible cube-like $n$-fold category.

Note that an R-unital $1$-actad is the same thing as a based operad. $1$-morphisms in $\mathscr{B}_2$ together with
$?\circ^2_i 1_1^2$, $?\circ^2_i 1_0^1$ (with the obvious compatibility relations, noting that $B_2^{R,0}=\{1_1^2, 1_0^1\}$,) generate a larger
category of trees, whose opposite is equivalent to a tree analogue of a category of finite sets and injections.

Now we can define higher analogues of a based algebra over a based operad:
An {\em R-unital algebra} over an R-unital $n$-actad $\mathscr{C}$ is a fibered $(n-1)$-fold category
$$X\r\mathscr{B}_{n-1}^R$$
together with a section $(n-1)$-fold functor
$$\diagram
& X\dto\\
\mathscr{B}_{n-1}^{R,0}\urto\rto^{\subset}& \mathscr{B}_{n-1}^R\\
\enddiagram
$$
which satisfies equivariance and associativity (and unit) properties analogous to those described
in Section \ref{salg}, in addition to the R-unitality axiom, which states that
if 
$$(y,(x_1,\dots,x_m))\in S-Mor(\mathscr{C}\times_{T_{\mathscr{B}_{n-1}}} T_X)$$
with $x_i=\iota\in S-Mor(\mathscr{B}_{n-1}^{R,0})$, $S\subseteq\{1,\dots,n-1\}$,
then
$$\Theta(y,(x_1,\dots,x_m))=\Theta(\gamma_{n,i}(y,\iota ), (x_1,\dots,x_{i-1},x_{i+1},\dots, x_m)).$$
(recall that $\Theta$ denotes the functor mapping
$\mathscr{C} \times_{T_{B_{n-1}}} T_X \rightarrow T_X$ in \rref{ThetaCxTXX}).

\vspace{5mm}

\section{Monads}\label{sMonads}

\vspace{5mm}

A key point of the infinite loop space theory of May \cite{MayIterated} is the close connection between operads and monads.
A monad (for an introduction see \cite{Bor}, Vol. 2, Ch.4) in a category $A$ is a functor $C:A\r A$ which
satisfies monoid-like properties with respect to composition. Its {\em algebra} is an object $X$ together with a natural transformation
$\Theta:CX\r X$ which satisfies left module-like axioms with respect to composition.

For a universal algebra in the category of sets, its associated monad $C$ assigns to a set $X$ the free algebra $CX$
of the given type on $X$. For an operad $\mathscr{C}$ (in the unbased context), the associated monad $C$ has a particularly simple form
\cite{MayIterated}:
$$CX=\coprod_{n\geq 1}\mathscr{C}(n)\times_{\Sigma_n}X^n$$
(on the right hand side, the subscript denotes taking orbits).

In this section, we will describe the monads associated to $n$-actads and plain $n$-actads. Let us discuss the plain case first. Let
$\mathcal{C}$ be a plain $n$-actad. In this case, we define a monad
$C$ in the category of sets over $B_{n-1}$. (In particular, we are dealing with
multisorted algebra in the sense of \cite{multi}.)
Consider a set
$\Xi :X\r B_{n-1}$ over $B_{n-1}$. Again, we write, for $y\in B_{n-1}$, $X(y)$ for $\Xi^{-1}(y)$. Then, for $y\in B_{n-1}$, 
$$CX(y):=\coprod_{z\in B_n, G_n(z)=y}\;\;\prod_{i=1}^{m_z}X(f_i(z)).$$

In the case of an $n$-actad $\mathscr{C}$, we will construct the monad $C$ describing $\mathscr{C}$-algebras in the category of $(n-1)$-fold categories fibered over $\mathscr{B}_{n-1}$. We could, alternatively, describe a monad in sets over $B_{n-1}$, considering the $\{ i \}$-morphisms
as unary operations. While this point of view also seems interesting, we do not follow it in the present paper.

\vspace{5mm}

Suppose $\mathscr{C}$ is an $n$-actad.
Let $X$ be an $(n-1)$-fold category with a fibered $(n-1)$-fold functor 
$$\Xi:X\r \mathscr{B}_{n-1}.$$ 
(Thus, $X$ is an invertible cube-like $(n-1)$-fold category.)
Now we will construct the free $\mathscr{C}$-algebra $CX$ on $X$.
There are $n$-fold fibered functors
$$T_{\Xi}:T_X\r T_{\mathscr{B}_{n-1}}$$
$$\Lambda:\mathscr{C}\r T_{\mathscr{B}_{n-1}}.$$
We can get a $\Lambda$ since there is an $n$-fold fibered functor
$$\mathscr{C}\r\mathscr{B}_n,$$
and we can compose with
$$\mathscr{F}_n:\mathscr{B}_n\r T_{\mathscr{B}_{n-1}}.$$
Suppose $S\subseteq \{1,\dots,n-1\}$. Then first define an $(n-1)$-fold category $C_0X$ with
an $(n-1)$-fold functor
$\Phi:C_0X\r\mathscr{B}_{n-1}$ by
$$S-Mor(C_0X)=S-Mor(\mathscr{C}\times_{T_{\mathscr{B}_{n-1}}}T_X)/$$
$$/(x,y)\sim (\widetilde{f}(x),\widehat{f}(y)), \text{ where } x\in S-Mor(\mathscr{C}),\; y\in S-Mor(T_X)\;$$
$$T_\Xi(y)=\Lambda(x)=\alpha\in T_{\mathscr{B}_{n-1}},\; (f:\alpha\r\beta)\in S-Mor(T_{\mathscr{B}_{n-1}}),$$
$$\text{for some }\beta, \text{ and } \widetilde{f} \text{ is the lift of } f \text{ using } \Lambda, \text{ and}$$
$$\widehat{f} \text{ is the lift of } f \text{ using } T_{\Xi}.$$
The trouble is that the $(n-1)$-fold functor $\Phi$ satisfies the
uniqueness, but not the existence properties of a fibered $(n-1)$-fold functor, due
to the same property of the $(n-1)$-fold functor
$$\mathscr{G}_n:\mathscr{B}_n\r\mathscr{B}_{n-1}.$$
Let $\mathscr{A}$ be a $k$-fold category.
The forgetful functor
$$\mathscr{U}_{\mathscr{A}}:\text{fibered }k-\text{fold categories over }\mathscr{A}\r k-\text{fold categories over }\mathscr{A}$$
(on both sides, $k$-fold functors are morphisms) has a left adjoint $K_{\mathscr{A}}$. 

\begin{definition}
For $k=n-1$, $\mathscr{A}=\mathscr{B}_{n-1}$,
the monad $C$ in $(n-1)$-fold categories fibered over $\mathscr{B}_{n-1}$ associated with an $n$-actad $\mathscr{C}$
is given by
$$CX=K_{\mathscr{B}_{n-1}}C_0X.$$
\end{definition}

Then $CX$ as defined is the free $\mathscr{C}$-algebra on $X$.

\vspace{5mm}

Now, one can also make a generalization of this construction to iterated algebras.
Suppose $\mathscr{C}$ is an $n$-actad. Then we have that $\mathscr{C}^{\alpha_1,\dots,\alpha_k}$ is a multisorted $n$-actad over 
$\{1,\dots,k\}$ (by which we mean a $n$-actad in the category of sets over $\{1,\dots,k\}$).
From this point of view, we can just repeat the above construction over $\{1,\dots,k\}$.

More explicitly, let $X=(X_1,\dots,X_k)$ be a $k$-tuple of $(n-1)$-fold categories with a fibered $(n-1)$-fold functor
$$\Xi_j:X_j\r \mathscr{B}_{n-1}.$$
Now let $X$ also be an $(n-1)$-fold category, where the objects are $k$-tuples of objects of $X_j$'s, in the same order, and the $S$-morphisms are $k$-tuples of $S$-morphisms of the $X_j$'s also in the same order.
So, we also have an $(n-1)$-fold fibered functor
$$\Xi=(\Xi_1,\dots,\Xi_k):X\r\mathscr{B}_{n-1}\times\{1,\dots, k\}.$$
(Thus, $X$ is an invertible cube-like $(n-1)$-fold category.)
Now, we will define the free $\mathscr{C}^{\alpha_1,\dots,\alpha_k}$ category $C^{\alpha_1,\dots,\alpha_k}X$ on $X$.

Then, there are $n$-fold, fibered functors
$$T_{\Xi}:T^k_X\r T^k_{\mathscr{B}_{n-1}\times\{1,\dots,k\}}$$
$$\Lambda:\mathscr{C}^{\alpha_1,\dots,\alpha_k}\r T_{\mathscr{B}_{n-1}\times\{1,\dots,k\}}^k.$$
The construction $\Lambda$ is as follows.
First, by definition, we have that
$$\mathscr{C}^{\alpha_1,\dots,\alpha_k}=\mathscr{C}\times_{\mathscr{B}_n}\mathscr{B}_n^{\alpha_1,\dots,\alpha_k}.$$
Since we have that $\mathscr{C}$ is an $n$-actad in the normal sense, we have a fibered $n$-fold functor
$$\mathscr{C}\r\mathscr{B}_n.$$

\vspace{5mm}

Then, we would need a fibered $m$-fold functor for every $m$, for every $k\leq m$ with some $1\leq\alpha_1<\dots<\alpha_k\leq m$,
$$h_m: \mathscr{B}_m^{\alpha_1,\dots,\alpha_k}\r\mathscr{B}_m\times\{1,\dots,k\}:$$
For $m=1$, we have that $k$ must equal $1$. In addition, we also must have $\alpha_1=1$.
However, for $k=1$, we have that $\mathscr{B}_{1}^{1}=\mathscr{B}_1$. Then we have the identity, which is a fibered functor,
$$\mathscr{B}_m^{\alpha_1,\dots,\alpha_k}\r\mathscr{B}_m = \mathscr{B}_m\times\{1,\dots,k\}.$$
Suppose are given this data for $m-1$. Fix some $1\leq\alpha_1<\dots<\alpha_k\leq m$.
If $\alpha_k < m-1$, then we have
\beg{hmfirststepsmalllastalpha}{\diagram
\mathscr{B}_m^{\alpha_1 , \dots , \alpha_k} \rto^{\mathscr{F}_m^{\alpha_1,\dots , \alpha_k}} & T_{\mathscr{B}_{m-1}^{\alpha_1, \dots , \alpha_k}}
\rto & T_{\mathscr{B}_{m-1}} \times \{1,\dots ,k\}\\ 
\enddiagram}
(where the second functor of \rref{hmfirststepsmalllastalpha} is given by applying $h_{m-1}$ to the entries of
$T_{\mathscr{B}_{m-1}^{\alpha_1, \dots , \alpha_k}}$ and taking the maximal element of $\{1, \dots ,k\}$ for the second coordinate,
as in \rref{maxwpxi}).
If $\alpha_k = m$, then we have
$$\diagram
\mathscr{B}_M^{\alpha_1, \dots, \alpha_{k-1}, n} \rto^{\mathscr{F}_m^{\alpha_1,\dots, \alpha_{k-1}, n}} & T^\bullet_{\mathscr{B}_{m-1}^{\alpha_1,\dots, \alpha_{k-1}}}.\\
\enddiagram
$$
Now, by Definition \rref{DefinitionPointedTMultisorted}, we have
$$T_{\mathscr{B}_{m-1}^{\alpha_1, \dots , \alpha_{k-1}}} \rightarrow \{1, \dots k\}$$
and we can construct
$$T_{\mathscr{B}_{m-1}^{\alpha_1, \dots, \alpha_{k-1} }}\rightarrow T_{\mathscr{B}_{m-1}^{\alpha_1, \dots , \alpha_{k-1}}} \rightarrow
T_{\mathscr{B}_{n-1}}$$
where the first map is the forgetful functor and the second map is induced by the natural functor
$$\mathscr{B}_{m-1}^{\alpha_1, \dots, \alpha_{k-1}} \rightarrow \mathscr{B}_{m-1}.$$
This, we can construct a fibered functor
\beg{hmfirststeplargeastalpha}{\mathscr{B}_m^{\alpha_1, \dots , \alpha_k} \rightarrow T_{\mathscr{B}_{m-1}} \times \{1, \dots, k\}}
Then we can use the fact that 
$$\mathscr{F}_m:\mathscr{B}_m\r T_{\mathscr{B}_{m-1}}.$$
is a fibered functor, to get a lift of every $S$-morphism of $T_{\mathscr{B}_{m-1}}$ to $\mathscr{B}_{m}$. Using this lifting by composing it with the 
fibered functor of \rref{hmfirststepsmalllastalpha} or \rref{hmfirststeplargeastalpha}, we get a fibered functor
$$h_m:\mathscr{B}_m^{\alpha_1,\dots,\alpha_k}\r \mathscr{B}_m\times\{1,\dots,k\}.$$
(This concludes the inductive definition of $h_m$.)

\vspace{5mm}

At the moment, we use this for $m=n-1$.
We have
$$h_{n-1}:\mathscr{B}_{n-1}^{\alpha_1,\dots,\alpha_k}\r \mathscr{B}_{n-1}\times\{1,\dots,k\}.$$
Then we have
$$
\diagram
\mathscr{C}^{\alpha_1,\dots,\alpha_k}\rto&\mathscr{B}_n\times_{\mathscr{B}_n}\mathscr{B}_n^{\alpha_1,\dots,\alpha_k}=\mathscr{B}_n^{\alpha_1,\dots,\alpha_k}\\
\enddiagram
$$
Then, we will compose with $\mathscr{F}_n^{\alpha_1,\dots,\alpha_k}$ to get a map
$$\mathscr{B}_n^{\alpha_1,\dots,\alpha_k}\r T_{\mathscr{B}_{n-1}^{\alpha_1,\dots,\alpha_k}}\r T_{\mathscr{B}_{n-1}\times\{1,\dots,k\}}^k.$$
Using this, we get our $\Lambda$.

\vspace{5mm}

We still want the free iterated $\mathscr{C}$-algebra on $X$ which we denote by $C^{\alpha_1,\dots,\alpha_k}X$ to be an $(n-1)$-fold category.
So, again, suppose $S\subseteq \{1,\dots,n-1\}$. Then define first $C_0^{\alpha_1,\dots,\alpha_k}X$ by
$$S-Mor(C_0^{\alpha_1,\dots,\alpha_k}X)=S-Mor(\mathscr{C}^{\alpha_1,\dots,\alpha_k}\times_{T^k_{\mathscr{B}_{n-1}\times\{1,\dots,k\}}}T^k_X)/$$
$$/(x,y)\sim (\widetilde{f}(x),\widehat{f}(y)), \text{ where } x\in S-Mor(\mathscr{C}^{\alpha_1,\dots,\alpha_k}),\; y\in S-Mor(T^k_X)\;$$
$$T_\Xi(y)=\Lambda(x)=\gamma\in T_{\mathscr{B}_{n-1}},\; (f:\gamma\r\delta)\in S-Mor(T^k_{\mathscr{B}_{n-1}\times\{1,\dots k\}}),$$
$$\text{for some }\delta, \text{ and } \widetilde{f} \text{ is the lift of } f \text{ using } \Lambda, \text{ and}$$
$$\widehat{f} \text{ is the lift of } f \text{ using } T^k_{\Xi}.$$
Then we have
$$C^{\alpha_1,\dots,\alpha_k}X=K_{\mathscr{B}_{n-1}\times\{1,\dots,k-1\}} C_0^{\alpha_1,\dots,\alpha_k}X.$$

Now these are analogous unital and R-unital versions of these concepts imposing additional identifications for the units,
analogous to the concept of based algebras over an operad \cite{MayIterated}. We omit the details.

\vspace{5mm}

\section{Case study of a $2$-actad algebra}\label{snew}

\vspace{5mm}

When applying operads to topology, we do not, of course, work in the base category of sets, but rather in topological spaces or simplicial sets.
All our definitions in this setting work essentially without change.
In particular, geometric realizations of fibered $n$-fold functors are still fibered.
In the case of spaces, we require all structure maps to be continuous.
In the case of simplicial sets, we can think of a simplicial object in any of the structures we considered so far,
which is considered as a category where morphisms are homomorphisms
(i.e. maps preserving all structure).
This makes sense, since all the structures we considered can be axiomatized as
(multisorted) universal algebras.
In particular, functors of singular sets and geometric realizations linking the structures in the topological and simplicial context
exist and behave in the usual way.
One source of examples is the \v{C}ech resolution.
For a space $X$, $EX$ denotes the simplicial space
$$EX:\Delta^{Op}\r \text{\em{Spaces}},$$
with $EX_n=X^{n+1}$ (faces are given by taking an $(n+1)$-tuple of elements of $X$ to an $n$-tuple that is the original $(n+1)$-tuple missing a coordinate, and degeneracies are given by taking an $n$-tuple to an $(n+1)$-tuple which is the same except with one of the coordinates copied twice). We often identify $EX$ with its geometric realization $|EX|$.

The \v{C}ech resolution tends to preserve algebraic structures, since the geometric realization of simplicial spaces preserves
finite products  (at least when we are working in the compactly generated category \cite{Steenrod}).
For example, considering the associative $n$-actad we introduced in Section \ref{salg}, we obtain the $n$-actad
$E\mathscr{A}_n$, which we call the {\em $E_{\infty}$-$n$-actad}.
For example, $E\mathscr{A}_1$ is the Barratt-Eccles operad \cite{MayEinfty}.
As shown in \cite{MayEinfty,MayIterated}, the classifying space $BA$ of a permutative category $A$ is an $E\mathscr{A}_1$-algebra, hence an $E_{\infty}$-space.
(Recall that the classifying space (or nerve) of a category is the simplicial space where $BA_n$ consists of composable $n$-tuples of morphisms in $A$,
meaning objects for $n=0$, faces are given by compositions or forgetting the first or last morphism, and degeneracies are insertions of identities
\cite{MayIterated}). For a group $G$ considered as a category with one object, we have an identification
$$BG\cong EG/G$$
where $G$ acts diagonally on the right hand side. The bijection is given by homogenization of coordinates
$$(g_1,\dots,g_n)\mapsto (1,g_1,g_1g_2,\dots,g_1\cdots g_n).$$
A permutative category is a symmetric monoidal category that is strictly associative
where the commutativity isomorphism is a strict involution.

Recall the $n$-actad $\mathscr{A}_n$ described in Example 1 at the end of Section \ref{salg}.
One may therefore ask what interesting $E\mathscr{A}_n$-algebras there are, what they mean in homotopy
theory, and how they arise. In this paper, I give some examples for $n=2$.
Note, however, that for any $n$, inclusion of the simplicial $0$-stage gives a morphism of $n$-actads
\beg{AntoEAn}{\mathscr{A}_n\r E\mathscr{A}_n.}
Therefore, $E\mathscr{A}_n$-algebras are, in particular, $(n-1)$-actads, and we can study their algebras. For example,
$E\mathscr{A}_2$-algebras are, in fact, operads.

In this section, we will study the free $E\mathscr{A}_2$-algebra on a given groupoid $X\r\mathscr{B}_1$,
and in the simplest non-trivial case of $X$ (the free ``binary" case) will give a condition on a category which guarantees that its classifying space is an algebra over the free $E\mathscr{A}_2$-algebra $A_{E\mathscr{A}_2}X$ on $X$.
In some sense, this can be thought of as an $E\mathscr{A}_2$ analogue of the results of \cite{MayEinfty} on the structure of the
Barratt-Eccles operad.

In a way, we can think of the operad $A_{E\mathscr{A}_2}X$, for a groupoid $X$, as describing algebras with operations from
$X$ which have commutativity properties given by $X$. However, these operations do not have any associativity properties. 
Arranging these operations in a planar tree (as in the reverse Polish notation on some old calculators), the $E\mathscr{A}_2$-algebra
property says that the order of the operations ``does not matter in the $E_{\infty}$-sense."
We see from this example that,
even for $n=2$, which is the next stage beyond the well-studies case of $n=1$, the world of $n$-actads becomes very rich.

\vspace{3mm}

To say things more precisely, suppose $\mathscr{A}$ is a 2-actad. Then suppose $X$ is a groupoid with a fibered functor
$$X\r \mathscr{B}_1.$$
Then, as noted in Section \ref{sMonads} above, the free $\mathscr{A}$-algebra on $X$ is the operad
$$A_{\mathscr{A}}X(n)=\Sigma_n\times_{1-Mor(\mathscr{B}_2)}{\coprod_{T\in B_2,m_{G_2(T)}=n}}(\mathscr{A}(T)\times_{Aut_2(T)}\prod_{i=1}^{m_T}X(f_i(T))).$$
For $T\in B_2=Obj(\mathscr{B}_2)$, we have that $E\mathscr{A}_2(T)=\{T\}\times Aut_2(T)$
(where $Aut_2(T) = \{2\} -Mor(T,T)$).
So, for every $X$,
$$A_{E\mathscr{A}_2} X(n)=$$
$$\Sigma_n\times_{1-Mor(\mathscr{B}_2)}\coprod_{T\in B_2,m_{G_2(T)}=n}\{T\}\times Aut_2(T)\times_{Aut_2(T)}\prod_{i=1}^{m_T}X(f_i(T))=$$
$$\Sigma_n\times_{1-Mor(\mathscr{B}_2)}\coprod_{T\in B_2,m_{G_2(T)}=n}\{T\}\times\prod_{i=1}^{m_T}X(f_i(T))=$$
$$\Sigma_n\times_{1-Mor(\mathscr{B}_2)}\coprod_{T\in B_2,m_{G_2(T)}=n}\prod_{i=1}^{m_T}X(f_i(T)).$$

As an example of this, let, for a set $S$, $X_2^S$ be the groupoid fibered over $\mathscr{B}_1$
$$X_2^S(n)= \emptyset, \text{ if } n \neq 2,$$
and let
$$X_2^S(2)=S\times\Sigma_2$$
be the free $\Sigma_2$-set on $S$.
Then let $\mathscr{C}$ be the underlying operad of the free $E\mathscr{A}_2$-algebra on $X_2^S$.
Then 
$$\mathscr{C}(n)\cong B_2(n-1)_{\text{binary}} \times\Sigma_n\times B(\Sigma_{n-1}^S),$$
where $\Sigma_{n-1}^S$ is the groupoid given by $\Sigma_{n-1}$ acting on $S^{n-1}$,
and $B_2(n)_{\text{binary}}$ are the elements of $B_2(n)$ which are of the form
$$((2,2,\dots,2),(i_1,i_2,\dots,i_{n-1}))$$
for some $n$.
(Note that $B_2(n)_{\text{binary}}$ can also be interpreted as the set of binary trees with $n+1$ leaves.)

\vspace{5mm}

Let $Y$ be a groupoid with a map
$$\Xi:S\r \text{Funct}(Y\times Y,Y).$$
Let $proj_i^2:S\times S\r S$ be the coordinate projections for $i\in\{1,2\}$,
$proj_i^3:S\times S\times S\r S$ be the coordinate projections for $i\in\{1,2,3\}$.
When the source is obvious, we will simply write $proj_i$.
We also have maps
$$\chi_{\ell}:S\times S\r \text{Funct}(Y\times Y\times Y,Y)$$
$$\chi_{r}:S\times S\r \text{Funct}(Y\times Y\times Y,Y)$$
given by $\chi_{\ell}=(\Xi\circ proj_2)\circ((\Xi\circ proj_1)\times Id)$ and
$\chi_r=(\Xi\circ proj_1)\circ (Id\times(\Xi\circ proj_2))$.
Note that we apply the argument only to $proj_2$ and $proj_1$, e.g.
$$\chi_\ell ((s,t)) = (\Xi (t)) \circ (\Xi (s) \times Id)$$
(we will continue to use this convention later).
Then let $T:Y\times Y\r Y\times Y$ be the functor that switches coordinates.
We also have maps induced by permutations of operations
$$\chi_{\ell,\ell},\chi_{r,\ell},\chi_{\ell,r},\chi_{r,r}, \chi_u:S\times S\times S\r \text{Funct}(Y\times Y\times Y\times Y,Y)$$
which are given by
$$\chi_{\ell,\ell}=(\Xi\circ proj_3)\circ(\Xi\circ proj_2 \times Id)\circ(\Xi\circ proj_1\times Id\times Id)$$
$$\chi_{r,\ell}=(\Xi\circ proj_3)\circ(\Xi\circ proj_2\times Id)\circ(Id\times\Xi\circ proj_1\times Id)$$
$$\chi_{\ell,r}=(\Xi\circ proj_3)\circ(Id\times\Xi\circ proj_2)\circ(Id\times\Xi\circ proj_1\times Id)$$
$$\chi_{r,r}=(\Xi\circ proj_3)\circ(Id \times\Xi\circ proj_2)\circ(Id \times Id\times \Xi\circ proj_1)$$
$$\chi_{u}=(\Xi\circ proj_2)\circ((\Xi\circ proj_1)\times(\Xi\circ proj_3)).$$
Consider the following additional structure for a groupoid $Y$:
\begin{enumerate}
\item We have isomorphisms
\beg{corralphaleft}{\alpha_{\ell}:\chi_{\ell}\r\chi_{\ell}\circ T,}
\beg{corralpharight}{\alpha_r:\chi_r\r\chi_r\circ T}
with the following commuting diagrams:

\vspace{5mm}

\begin{tikzcd}
\chi_{\ell}\arrow[ddr,"Id"]\arrow[r,"\alpha_{\ell}"]& \chi_{\ell}\circ T\arrow[d, "\alpha_{\ell}\circ T"]\\
& \chi_{\ell}\circ T\circ T\arrow[d,"="]\\
&\chi_{\ell}\\
\end{tikzcd}
\begin{tikzcd}
\chi_{r}\arrow[ddr,"Id"]\arrow[r,"\alpha_{r}"]& \chi_{r}\circ T\arrow[d, "\alpha_{r}\circ T"]\\
& \chi_{r}\circ T\circ T\arrow[d,"="]\\
&\chi_{r}\\
\end{tikzcd}

\item
The following diagram (and the ones similar to it, with a different $\chi_{?,?}$'s) commute:
$$
\diagram
\chi_{\ell,\ell}\rto^{Id}\dto_{(\Xi\circ proj_3)(\alpha_{\ell}\times Id)}& \chi_{\ell,\ell}\dto^{\alpha_{\ell}\circ (\Xi\circ proj_1)\times Id\times Id}\\
\chi_{\ell,\ell}\circ(T\times Id)\dto_{\alpha_{\ell}\circ(\Xi\circ proj_1)\times Id\times Id}& \chi_{\ell,\ell}\circ(Id\times T)\dto^{(\Xi\circ proj_3)(\alpha_{\ell}\times Id)}\\
\chi_{\ell,\ell}\circ(Id\times T)\circ(T\times Id)\dto_{(\Xi\circ proj_3)(\alpha_{\ell}\times Id)}&\chi_{\ell,\ell}\circ(T\times Id)\circ(Id\times T)\dto^{\alpha_{\ell}\circ(\Xi\circ proj_1)\times Id\times Id}\\
\chi_{\ell,\ell}\circ(T\times Id)\circ(Id\times T)\circ(T\times Id)\rto^=& \chi_{\ell,\ell}\circ(T\times Id)\circ(Id\times T)\circ(T\times Id)\\
\enddiagram
$$

\item
The following diagram commutes:
$$
\diagram
\chi_u\rto^{Id}\dto_{\alpha_{\ell}\circ(Id\times Id\times(\Xi\circ proj_3))}&\chi_u\dto^{\alpha_r\circ((\Xi\circ proj_1)\times Id\times Id)}\\
\chi_u\circ(T\times Id)\dto_{\alpha_r\circ((\Xi\circ proj_1)\times Id\times Id)}&\chi_u\circ(Id\times T)\dto^{\alpha_{\ell}\circ(Id\times Id\times(\Xi\circ proj_3))}\\
\chi_u\circ(Id\times T)\circ(T\times Id)\dto_{\alpha_{\ell}\circ(Id\times Id\times(\Xi\circ proj_3))}& 
\chi_u\circ(T\times Id)\circ(Id\times T)\dto^{\alpha_r\circ((\Xi\circ proj_1)\times Id\times Id)}\\
\chi_u\circ(T\times Id)\circ(Id\times T)\circ(T\times Id)\rto^{=}&\chi_u\circ(Id\times T)\circ\circ(T\times Id)\circ(Id\times T).\\
\enddiagram
$$
\end{enumerate}

Then we claim the following
\begin{theorem}\label{BYalg}
Given the structure described the above axioms $(1)$-$(3)$,
the classifying space on morphisms $BY$ is an algebra over the operad $A_{E\mathscr{A}_2}(X^S_2)$.
\end{theorem}

\vspace{5mm}

The Theorem will be proved at the end of this Section.
To prove the Theorem, we first will prove some results about $B_2(n)_{\text{binary}}$.

\vspace{5mm}

As a warm-up case, consider the usual presentation of the symmetric group:
\beg{SigmanPresentationCase}{\Sigma_n\cong \langle a_1,\dots, a_{n-1}|a_i^2,\; a_ia_{i+1}a_ia_{i+1}a_ia_{i+1},\; a_ia_ja_ia_j\; j>i+1\rangle.}
Here $a_i$ are understood as switches of consecutive terms in an $n$-element sequence.
This can be proved for $4$ directly from using the Cayley graph. Then, one uses induction and the isotropy groups of elements of $\{1,\dots,n\}$ in the standard action of $\Sigma_n$ to prove it for every $n$.

\vspace{5mm}

In the case of an element of $B_2(n)_{\text{binary}}$, the terms are the triangles (we shall say ``nodes") of a binary tree.
Thus, they do not naturally form a sequence but could be thought of as forming themselves a tree where edges are between triangles which are attached.
For a given binary tree, we will present the group on its nodes in terms of generators which correspond to such edges
(of which there are $n-1$).

Concretely, we claim
\begin{theorem}\label{Theorempresentation}
We have
\beg{B2nBinary2Morphism}{
\begin{array}{c}2-Mor(\mathscr{B}_2(n)_{\text{binary}})\cong\langle a_1,\dots,a_{n-1}|a_i^2,\\
a_ia_{j}a_ia_{j}a_ia_{j} \text{ when $i\neq j$ and $a_i$ and $a_j$ share a node},\\
(a_ia_{j}a_{k}a_{j})^2\text{ when $i<j<k$ and $a_i$, $a_j$, $a_k$ share a node},\\
a_ia_ja_ia_j \text{ when $a_i$ and $a_j$ do not share a node}\rangle.
\end{array}}
(Here we denote by $2-Mor(\mathscr{B}_2(n)_{\text{binary}})$ the full subcategory of $2-Mor(\mathscr{B}_2)$ on binary trees.)
\end{theorem}

\noindent {\bf Comment:} When $a_i$, $a_j$, $a_k$ share a node with $i,j,k$ different, the relation
$(a_ia_j a_k a_j)^2$ follows from the relations \rref{B2nBinary2Morphism} regardless of the order of $i,j,k$.
This is by the fact that the relations \rref{B2nBinary2Morphism} involving $a_i$, $a_j$, $a_k$ generate a copy of $\Sigma_4$, which follows
by examining the Cayley graph.

\begin{proof}

If $n=4$, this can be proved drawing the Cayley graph.
For $n\geq 5$, suppose we have a binary tree
$$((x_1,\dots,x_k),(i_1,\dots,i_{k-1}))=((2,\dots,2),(i_1,\dots,i_{k-1})).$$
For lower $n$, assume the claim as an induction hypothesis.
Then suppose we can
fix $x_{i-1}$, $x_i$, $x_{i+1}$, $x_{i+2}$, $x_{i+3}$ with $x_i$ connected to both $x_{i+1}$ and $x_{i+2}$.
Then let $a, b, c, d$ be the $2$-morphisms that switch $x_{i-1}$ and $x_i$, $x_i$ and $x_{i+2}$, $x_{i+3}$ and $x_{i+2}$, and $x_{i}$ and $x_{i+1}$ 
respectively. In a picture, this looks like

\vspace{5mm}

\beg{TheoremPresentationBinary5tree}{\begin{tikzpicture}
\draw [fill] (0,2) circle [radius=0.05];
\draw [fill] (2,2) circle [radius=0.05];
\draw [fill] (2,0) circle [radius=0.05];
\draw [fill] (4,2) circle [radius=0.05];
\draw [fill] (6,2) circle [radius=0.05];
\draw [<->] (0,2)--(2,2);
\draw [<->] (2,2)--(2,0);
\draw [<->] (2,2)--(4,2);
\draw [<->] (4,2)--(6,2);
\draw (1,2.25) node {a};
\draw (1.75,1) node {d};
\draw (3,2.25) node {b};
\draw (5,2.25) node {c};
\draw [<->] (4,2)--(2,0);
\draw (3.52,1) node {bdb};
\draw (0,2.25) node {$x_{i-1}$};
\draw (2,2.25) node {$x_{i}$};
\draw (2,-0.5) node {$x_{i+1}$};
\draw (4,2.25) node {$x_{i+2}$};
\draw (6,2.25) node {$x_{i+3}$};
\end{tikzpicture}
}

\vspace{5mm}
\noindent where the dots represent the $x_j$'s.
Let 
\beg{TheoremPresentationGiDefinition}{\begin{array}{c}
G_i=\langle a, b,c,d|a^2,b^2,c^2,d^2,\\ 
(ab)^3, (bc)^3,(ad)^3,(db)^3,(daba)^2, (ac)^2, (cd)^2\rangle
\end{array}}
(in our notation here, we rename the $x_i$'s after each switch, making the above compositions possible).
Then let $G_{i,1}$ be the subgroup of $G_i$ that is generated by $a,b,d$. Let $G_{i,2}$ be the subgroup of $G_i$ that is generated by $bdb,b,c$.
$$\kappa_i:G_{i,1}\r G_{i,2}$$
with 
\beg{Kappaidescription}{\begin{array}{c}\kappa_i(a)=b\\
\kappa_i(b)=c\\
\kappa_i(d)=bdb.
\end{array}}
So, in other words, $\kappa_i$ shifts $x_i$ to $x_{i+2}$ and takes the permutations around $x_i$ to those around $x_{i+2}.$

First, note that by the case $n=4$, $G_{i,1}$ is the symmetric group, and also can be presented only in the relations
\rref{TheoremPresentationGiDefinition} including $a$, $b$, $d$.
We want $\kappa_i$ to be a isomorphism, so we have to check the relations of $G_i$ for $a,\; b,$ and $d$ for $\kappa_i(a),\;\kappa_i(b),$
and $\kappa_i(d).$
It is obvious that $(\kappa_i(a))^2=b^2,\;(\kappa_i(b))^2=c^2$ are $1$.
We have
$$(\kappa_i(d))^2=(bdb)^2=bddb=bb=1.$$

Next, we have
$$(\kappa_i(a)\kappa_i(d))^3=(bbdb)^3=(db)^3=1.$$

Next, we can commute $d$ and $c$ since they permute separate pairs of $x_j$'s.
$$(\kappa_i(d)\kappa_i(a)\kappa_i(b)\kappa_i(a))^2=$$
$$bdbbcbbdbbcb=bdcdcb=bddccb=bb=1.$$

Now we need to check the Yang-Baxter relations. It is obvious that
$$\kappa_i(a)\kappa_i(b)\kappa_i(a)\kappa_i(b)\kappa_i(a)\kappa_i(b)=bcbcbc=1.$$
We have
$$(\kappa_i(d) k_i(b))^3 = (bdbc)^3 = $$
$$=bd (bcbd)^2 bc =bd (cbcd)^2 bc = bd (cbdc)^2bc =$$
$$bdc(bd)^2cbc = bc(d(bd)^2)cbc = (bc)^3 =1$$
So, we have every relation. 
Thus, $\kappa_i$ is a
homomorphism. On the other hand, all of the relations \rref{TheoremPresentationGiDefinition} are valid in the symmetric group,
and thus the image of $\kappa_i$ must be a symmetric group and $\kappa_i$ must be an isomorphism.

\vspace{5mm}

Now suppose we have a binary tree in $B_2(n)_{\text{binary}}$.
Represent its triangles as nodes $x_i$, and edges $a_i$ between them where
the triangles are attached.
Consider a degree three node which is farthest from the root.
Using the move \rref{TheoremPresentationBinary5tree}, we can switch one edge in a way that this node moves farther from the root,
and the relations for the new tree are equivalent to the relations of the old tree.
This is because all the relations which change occur in presentations of $\Sigma_4$ subgroups, which stay
isomorphic, as we just proved.
(In the end, we also need a variant of the move \rref{TheoremPresentationBinary5tree} where the edge $c$ is absent,
but that case is clear since we already know the group is $\Sigma_4$).

By repeating these moves, the tree can be turned into a sequence, in which case we are reduced to \rref{SigmanPresentationCase}.

%
\end{proof}

\noindent {\em Proof of Theorem \ref{BYalg}.}
Analogous to the argument of \cite{MayEinfty} that for a permutative category $A$, $BA$ is an $E\mathscr{A}_1$-algebra.
Consider $n$ composable $p$-tuples of $Y$-morphisms
$$
\diagram
X_{1,0}\rto^{f_{1,1}}&X_{1,1}\rto^{f_{1,2}}& \dots\rto^{f_{1,p}}& X_{1,p}\\
&\vdots & \vdots & \\
X_{n,0}\rto^{f_{n,1}}&X_{n,1}\rto^{f_{n,2}}& \dots\rto^{f_{n,p}}& X_{n,p}.\\
\enddiagram
$$
Choose a tree $T\in B_2(n)_{binary}$ with $p+1$ decorations
$T_0,\dots, T_p$ of its nodes by elements of $S$. Apply the operations corresponding to $T_i$ to the objects $X_{1,i},\dots,X_{n,i}$, and apply the same operations, composed with the appropriate composition of the coherences \rref{corralphaleft},\rref{corralpharight}, to the morphisms
$f_{1,i},\dots,f_{n,i}$. The coherence diagrams $(1)$-$(3)$ (given in this section,
directly before the statement of Theorem \ref{BYalg}) guarantee consistency of these compositions, by Theorem \ref{Theorempresentation},
thereby describing an $A_{E\mathscr{A}}(X_2^S)$-action on $BA$.
We construct maps
$$A_{E\mathscr{A}_2} X_2^S (n) \times (BY)^n \rightarrow BY.$$
\qed

\vspace{5mm}

\section{Examples of New Operads}

\vspace{5mm}

In this section, I will describe a method for using $E\mathscr{A}_2$ algebras to construct new examples of operads.

\vspace{3mm}

First, suppose $X$ is an operad. Then we have a functor
$$X\r \mathbb{N},$$
(where the morphisms in $\N$ are identities)
since we have a projection functor $\mathscr{B}_1\r\mathbb{N}$.

Then define, for $m,n\in\N, k\in\N_0$,
$$
\begin{array}{c}
S_k(n,m)=\{(T_0,\dots,T_k)|T_j\in B_2,\; m=m_{T_j},\; G(T_j)=n,\\
\text{there exists permutations }\sigma_j:\{1,\dots,m\}\rightarrow\{1,\dots, m\}\\ 
\text{such that, for} \; j=1,\dots,k,\; f_i(T_0)=f_{\sigma_j(i)}(T_j)\}.\\
\end{array}
$$
This means that if $(T_0,\dots,T_k)\in S_k(n,m)$, then the $f_i(T_j)$'s are the same as the $f_i(T_0)$'s, up to permutation.
Then we can define monads $M_k$ in the category of spaces over $\mathbb{N}$ of the form
$$M_kX(n)=\coprod_m\coprod_{(T_0,\dots,T_k)\in S_k(n,m)}\underbrace{\Sigma_n\times\dots\times\Sigma_n}_{k+1}\times\prod_{i=1}^{m_{T_0}} X(f_i(T_0)),$$
for $k\in \mathbb{N}_0$. The $\Sigma_n$ factors correspond to permutations of ``prongs" of the
trees $T_0,\dots,T_k$ (as would arise in the free operad on $\mathscr{B}_1$; see Sections \ref{salg}, \ref{sMonads}).
Monad composition is defined by $2$-composition of trees following the order of the natural numbers $1,\dots,m_{T_0}$ in the $0$'th coordinate, and the order specified by the permutation on $S_k(n,m)$ in the other coordinates. This $2$-composition is ``twisted" by the
$1$-permutation of prongs in the $\Sigma_n$ factors (meaning that the successors are switched according to the permutation of the prongs of the inserted tree). The permutations in the resulting $S_k(n,m)$, as well as in the $\Sigma_n$ factors, are then determined as appropriate
versions of the wreath product (to match entries which were originally matched).
This is obviously associative. We observe that for $k=0$, algebras over this monad are just operads in the category of spaces, i.e. the monad takes
$\coprod X(n)\times\Sigma_n$, and then applies the monad defining operads in monoids fibered over $\mathscr{B}_1$, as considered in
Section \ref{salg}, \ref{sMonads}.

Now let
\beg{scrMkl1tomn}{\mathscr{M}_k(\ell_1,\dots,\ell_m;n)=\Sigma_m\times\coprod_{(T_0,\dots,T_k)\in S_k(n,m), \; \ell_i=m_{f_i(T_0)}}\underbrace{\Sigma_n\times\dots\times\Sigma_n}_{k+1}.}
Obviously, by counting prongs, this is only non-empty when 
\beg{nl1lmm1}{n=\ell_1+\dots\ell_m-m+1.}
Again, the $\Sigma_n$'s come from the permutation parts of the $T_i$'s. As above, we have $(k+1)$ copies of $\Sigma_n$ because we have $(k+1)$ $T_j$'s
encoding permutations of ``prongs." The extra $\Sigma_m$ is so we have $n$ (and not $n-1$) $2$-permutations.
We have, for every $k$,
$$\mathscr{M}_k=\underbrace{\mathscr{M}_0\times\dots\times\mathscr{M}_0}_{k+1},$$
since the $k$ extra $\Sigma_m$'s on the right hand side corresponds to the $S_k(n,m)$
(instead of $S_0 (n,m)$) on the bottom of the coproduct in \rref{scrMkl1tomn}.
Let
$$\mathscr{M}:=\mathscr{M}_0.$$

Now, $\mathscr{M}$ is an $\N$-sorted operad (or multi-category, see \cite{EM}) and, therefore, so is $\mathscr{M}_n$. In addition,
we see that $M_k$ is actually the monad associated
with $\mathscr{M}_k$ (Note, again, that the extra $\Sigma_m$ in \rref{scrMkl1tomn} prevents extra identifications). $\mathscr{M}_0$-algebra are exactly operads, which are $\mathscr{A}_2$-algebras.

In fact, every $\mathscr{M}$-algebra is the sequence of spaces (i.e., space over $\N$) of an 
$\mathscr{A}_2$-algebra (i.e. operads). 
Thus, every $E\mathscr{M}$-algebra
is the underlying sequence of spaces of an $E\mathscr{A}_2$-algebra. 
We should explain that this last statement is not `if and only if.' The reason is that the $2$-permutations in $\mathscr{M}_k$ (or $M_k$) can
occur between the ``triangles" of different trees, $T_i$, $T_j$.
While $m_{T_i}=m_{T_j}=m$, we have no functorial correspondence between the ``prongs" of the trees $T_i$, $T_j$.
Therefore, if we allowed these more general $2$-permutations in $2$-actads, there would be no consistent way of defining $\circ^2$.

Now suppose we have a spectrum $E$ (for example in the sense of May, e.g., see \cite{MayEinfty, MayIterated, MaySpectra, MayMultiplicativeA}, i.e. a sequence of spaces $E_n$, $n\in \N_0$, with homeomorphisms
$$E_n\cong \Omega E_{n+1}$$
where $\Omega E_{n+1}=Map_{based} (S^1, E_{n+1})$ denotes the based loop space).
Suppose we have a morphism of spectra
$$E\r H\Z.$$
Then we have a map
\beg{E0varphiZ}{
\diagram
E_0\rto^\varphi&\Z,\\
\enddiagram
}
with $E_0(n)=\varphi^{-1}(n-1)$.
Now $E_0$ is an algebra over $\mathscr{C}_{\infty}$, the $\infty$-little cubes operad.
Now, let $\mathscr{C}$ be an $\N$-sorted operad with
$$\mathscr{C}(\ell_1,\dots,\ell_m,n)=\mathscr{C}_{\infty}(m)$$
when \rref{nl1lmm1} holds, and
$\mathscr{C}(\ell_1,\dots,\ell_m,n)=\emptyset$ else.
Then $(E_0(n))_{n\in\N}$ is a $\mathscr{C}$-algebra.

So, we have projections
$$\diagram
\mathscr{C}\times E\mathscr{M}\dto_{p_1}\rto^(0.6){p_2} &E\mathscr{M}\\
\mathscr{C}.& \\
\enddiagram
$$
Then let $N$ be the monad associated with $\mathscr{C}\times E\mathscr{M}$.
By \rref{scrMkl1tomn}, the symmetric group action on $\mathscr{M}$ is free, so $E\mathscr{M}$ is an $E_\infty$-operad
since the \v{C}ech resolution of a non-empty space is always (simplicially) contractible.
Then the $2$-sided bar construction of monads
$$B(EM,N, (E_0(n))_{n\in \N})\simeq (E_0(n))_{n\in\N}$$
(see \cite{MayIterated})
is an $E\mathscr{M}$-algebra. So it is an $E\mathscr{A}_2$-algebra.
The same construction also works for a map of $E_{\infty}$-spaces $X\r\Z$ in place of $E_0$.
(Recall that infinite loop space of spectra are group-like $E_\infty$-spaces.)
Thus, we have
\begin{proposition}
Suppose
$$X\r\Z$$
is a map of $E_{\infty}$-spaces. Then there exists an $E\mathscr{A}_2$-algebra $\Xi_X$ equivalent to $(X(n))_{n\in\N}$.
\end{proposition}
\qed

\vspace{5mm}

We will now prove that in some cases, these $E\mathscr{A}_2$-algebras are non-trivial as operads (see \rref{AntoEAn}) in the sense that an $E_{\infty}$-operad does not map into them.

Let $X$ be the free $E\mathscr{A}_1$-algebra on $*$, with
$$X(n)=B\Sigma_{n-1}$$
(which comes with a natural $\Sigma_n$-action).
Now, we will write
$$X_n=X(n+1)=B\Sigma_n=E\Sigma_{n}\times_{\Sigma_{n}}*,$$
so that, in particular, $X_0=X(1)$.

\begin{theorem}
There does not exist an $E_{\infty}$-operad $\mathscr{C}$ together with an operad morphism
\beg{thmscrCX}{\mathscr{C}\r \Xi_X.}
\end{theorem}

\begin{proof}
Suppose such a $\mathscr{C}$ does exist.
Since $\Xi_X$ is an operad in spaces, we have a $\mathbb{Z}/2$-equivariant map
\beg{222-4}{\varsigma:\Xi_X(2)\times(\Xi_X(2)\times \Xi_X(2))\r \Xi_X(4),}
where on the right hand side,
$\Z/2\r\Sigma_4$ acts by sending the generator of $\Z/2$ to the permutation $(13)(24)$.
For $\Z/2$-spaces $Z$ and $T$, where $Z$ has a trivial $\Z/2$-action, we have an equivalence
$$Z\times_{\Z/2}(T\times E\Z/2)\simeq Z\times (E\Z/2\times_{\Z/2} T),$$
or equivalently
\beg{Z2actionequiv}{(Z\times E\mathbb{Z}/2) \times_{\mathbb{Z}/2} T \simeq Z \times (E\mathbb{Z}/2 \times_{\mathbb{Z}/2}T)}
By \rref{Z2actionequiv}, we have
$$\diagram
\varrho:(E\Z/2\times X_1)\times_{\Z/2}(X_1\times X_1)\rto^{\cong}&X_1\times(E\Z/2\times_{\Z/2}(X_1\times X_1)).\\
\enddiagram$$
We also have a map
$$\nu:E\Z/2\times_{\Z/2} X_3\r E\Z/2\times_{E\Z/2} X_3,$$
since, by construction, the $\Sigma_n$-action on $X(n)$ extends, for every $n$, to an $E\Sigma_n$-action.
Let
$$\mu=\nu\circ(E\Z/2\times\varsigma)\circ\varrho^{-1}.$$
So, we have the following diagram:

\vspace{3mm}

\begin{tikzcd}
(E\Z/2\times X_1)\times_{\Z/2}(X_1\times X_1)\arrow[d,"\varrho"]\arrow[r, "E\Z/2\times\varsigma"]& E\Z/2\times_{\Z/2} X_3\arrow[d,"\nu"]\\
X_1\times(E\Z/2\times_{\Z/2}(X_1\times X_1)) \arrow[rd,"\mu"] & E\Z/2\times_{E\Z/2}X_3 \arrow[d, "="]\\
& X_3
\end{tikzcd}

Additionally, by construction, the map $\mu$ comes from the $E_{\infty}$-algebra (concretely, $\mathscr{C}_{\infty}$-algebra) structure of $X$.
Now, take homology of $\mu$ with coefficients in $\Z/2$, i.e.
$$H_k(\mu): H_k(X_1\times(E\Z/2\times_{\Z/2}(X_1\times X_1)))\r H_k(X_3).$$
Now, let $\alpha$ be the generator of $H_0(X_1)=\Z/2$ and let $e_k$ be the generator of $H_k(B\Z/2)\cong\Z/2$.
Then we have
\beg{dyerlashofgenthing}{\alpha\otimes(e_k\otimes\alpha\otimes\alpha)\mapsto \alpha\cdot Q^k\alpha\neq 0.}
where $Q^k$ denotes the Dyer-Lashof operation (see \cite{MaySteenrod, LCM}).

\vspace{3mm}

Now, suppose we have a map of operads \rref{thmscrCX}. Then we have the commutative diagram
$$\diagram
\mathscr{C}(4)\rto\dto& X(4)\dto \\
E\Z/2\times_{\Z/2}\mathscr{C}(4)\rto&E\Z/2\times_{\Z/2} X(4)
\enddiagram
$$
commutes.
Then the diagram

\vspace{3mm}

\hspace{25mm}\begin{tikzcd}
\mathscr{C}(4)\arrow[r]\arrow[d]& X(4)\arrow[d]\\
E\Z/2\times_{\Z/2}\mathscr{C}(4) \arrow[rdd]\arrow[dd] \arrow[r]& E\Z/2\times_{\Z/2}X(4) \arrow[d]\arrow[lddd, controls={+(-3.5,-1) and +(5,0.8)}]\\
& E\Z/2\times_{E\Z/2}X(4)\arrow[d]\arrow[dd,bend left]\\
E\Sigma_4\times_{\Sigma_4}\mathscr{C}(4)\arrow[d]& X(4) \arrow[d]{d}{=}\\
E\Sigma_4\times_{\Sigma_4}X(4)\arrow[r] & E\Sigma_4\times_{E\Sigma_4}X(4)
\end{tikzcd}

\vspace{3mm}

commutes.
In particular, we have a diagram of the form

\vspace{3mm}

\hspace{30mm}\begin{tikzcd}
E\Z/2\times_{\Z/2}\mathscr{C}(4)\arrow[d]\arrow[r]& X(4)\\
E\Sigma_4\times_{\Sigma_4}\mathscr{C}(4), \arrow[ru] &
\end{tikzcd}

\vspace{3mm}

in other words,

\vspace{3mm}

\hspace{40mm}\begin{tikzcd}
B\Z/2 \arrow[d]\arrow[r]& B\Sigma_3\\
B\Sigma_4 \arrow[ru] &
\end{tikzcd}

\vspace{3mm}

commutes.
So, by taking $\pi_1$, we get a commutative diagram of groups

\vspace{3mm}

\hspace{40mm}\begin{tikzcd}
\Z/2 \arrow[d,"g"]\arrow[r,"f"]& \Sigma_3\\
\Sigma_4. \arrow[ru,"h"] &
\end{tikzcd}

\vspace{3mm}

By definition, $g$ takes the generator of $\Z/2$ to $\sigma=(13)(24)$, $f$ is an inclusion (by \rref{dyerlashofgenthing}), and $h$ is a group homomorphism.
Then $f=h\circ g$. We have that $\sigma=\sigma_1\circ\sigma_2$, where $\sigma_1=(13)$ and $\sigma_2=(24)$. Since $h$ is a group homomorphism, it 
takes $\sigma$ to $h(\sigma_1)\circ h(\sigma_2)$, and $h(\sigma_1)$, $h(\sigma_2)$ are conjugate in $\Sigma_3$, since $\sigma_1$, $\sigma_2$
are conjugate in $\Sigma_4$. Thus, $h(\sigma_1)$, $h(\sigma_2)$ are both non-trivial. However, a non-trivial involution in $\Sigma_3$ 
is not a composition of two non-trivial involutions.
Contradiction.
\end{proof}

Let me add a few words on the topological interpretation of $\Xi_X$-algebras where $\Xi_X$ is the operad associated with an $E_{\infty}$-space
$X=(X_n)_{n\geq 0}$. First, let us discuss the case when $X=E_0$ for a spectrum $E$ with a morphism of (connective) spectra
\beg{EHZ}{E\rightarrow H\mathbb{Z}.}
This gives a morphism of $E_{\infty}$-spaces
$\varphi:E_0\rightarrow \mathbb{Z}$, and as above, we adopt the convention
$$X(n)=X_{n-1}=\varphi^{-1}(n-1),$$
for $n\in\mathbb{Z}$.
In particular, $X_0$ is an $E_{\infty}$-space (over $\mathbb{Z}$)
which is the $0$-space of the fiber $\widetilde{E}$ of the morphism of spectra $\rref{EHZ}$.
While we set up our formalism in such a way that $0$ is not included in the indexing of an operad, completely analogously, one can also construct a
based operad $\bar{\Xi}_X$ indexed by $n\in\mathbb{N}_0$, with
$$\bar{\Xi}_X(0)\sim X_{-1}.$$
For $\Xi_X$-algebra $Y$, 
$Y\amalg X_{-1}$ is canonically an $\bar{\Xi}_X$-algebra.
Now we already mentioned the fact that when \rref{EHZ} splits, then $\bar{\Xi}_X$-algebras are equivalent to $E_{\infty}$-spaces
$Y$ together with $E_{\infty}$-maps
$$\widetilde{E}_0\rightarrow Y.$$
For a general morphism \rref{EHZ}, a splitting can be achieved by replacing $E$ with the homotopy pushout $\check{E}$ of the diagram
$$
\diagram
\widetilde{E}\dto\rto& E\\
E.& \\
\enddiagram
$$
Then, $X$ gets replaced with $\check{X}$ where for $n\in \mathbb{Z}$,
$\check{X}(n)=X=E_0$.
Then, a $\bar{\Xi}_{\check{X}}$-algebra $Y$ is equivalent to a morphism of $E_{\infty}$-algebra $X\rightarrow Y$.
However, if we denote the monads associated with the operads we considered by replacing the symbol $\Xi$ with $\Theta$,
then for a $\bar{\Xi}_X$-algebra $Y$,
\beg{CheckYdefn}{\check{Y}=B(\bar{\Theta}_{\check{X}},\bar{\Theta}_X,Y).}
gives a diagram of $E_{\infty}$-spaces
\beg{CheckYdiagr}{
\diagram
X\rto\drto& \check{Y}\dto^{\psi}\\
 & \mathbb{Z}.\\
\enddiagram
}
Moreover, we get a canonical equivalence
$$
\diagram
\psi^{-1}(-1)\rto^(0.6){\sim}& Y.\\
\enddiagram
$$
Thus, we see that $\bar{\Xi}_X$-algebras $Y$, where $\pi_0 Y$ is a group, are canonically identified, up to equivalence, with
fibers of $-1$ in diagrams of $E_{\infty}$-spaces of the form \rref{CheckYdiagr}.
Such spaces are, of course, homotopically equivalent to the fibers $Y'=\psi^{-1}(0)$, which are equivalent to $E_{\infty}$-morphisms
$$X_0\rightarrow Y'.$$
However, when \rref{EHZ} does not split, the homotopy equivalence $Y\simeq Y'$ is not canonical.

Finally, let us comment what happens when the $E_{\infty}$-space $X$ is not the $0$-space of a spectrum (i.e. is not group complete). Then
$\bar{\Xi}_X$ does not make sense as above, and as the $\Xi_X$-algebra $Y$ is necessarily unbased.
However, if $X^g$ is the infinite loop space associated with $X$ (i.e. its group completion) then applying $B(\Theta_{X^g},\Theta_X,?)$
in the above sense gets us into the above situation.
In particular, the $\pi_0$ of $Y_+=Y\amalg\{*\}$ is still canonically a monoid, and its group completion is canonically equivalent to $\psi^{-1}(-1)$
in the diagram \rref{CheckYdiagr}, where $X$ is replaced with $X^g$.

\vspace{5mm}

\section{Ordinals}\label{sord}

\vspace{5mm}

In this section, we will consider the plain non-unital $n$-base $B_n$ only.
Let $Ord$ denote the set of ordinals.
It is well known that ordinal numbers $0<\alpha<\varepsilon$ can be written uniquely as a sum
$$\alpha=\omega^{\alpha_1} +\dots+\omega^{\alpha_k},$$
for some $\alpha>\alpha_1\geq\dots\geq\alpha_k$
Now, by definition, every tree in $B_2$ can be written as
\beg{zgamilto1z0tol}{z=\gamma_{1,i_{\ell}}(\dots\gamma_{1,i_2}(\gamma_{1,i_1}(z_0,z_1),z_2)\dots,z_{\ell})}
with $i_1>\dots>i_\ell$, $z_0\in B_1=\N$, $z_1,\dots,z_{\ell}\in B_2$ (so all the $z_1,\dots,z_{\ell}$ are plugged directly
into $z_0$). So if we define inductively a function $\Phi_2: B_2\r Ord$
by $\Phi_2(n)=n$ for $n\in\N$,
$$\begin{array}{c}
\Phi_2(z)= (i_1-1) +\omega^{\Phi_2(z_1)} +(i_2-i_1-1) +\omega^{\Phi_2(z_2)}+\dots\\
+ (i_{\ell}-i_{\ell-1}-1) +\omega^{\Phi_2(z_{\ell})}+(z_0-i_{\ell}),\\
\end{array}
$$ 
then $\Phi_2$ is onto $\varepsilon$. We could identify a subset on $B_2$ on which $\Phi_2$ is bijective, but it will not be bijective on all of $B_2$ because of order of summation, and also non-cancellation, for example $1+\omega =\omega$.

\vspace{3mm}

It is therefore natural to investigate how this fact generalizes to $n>2$.
We will focus on the surjectivity part here. For $n>2$, we will introduce more arguments to the function
$$\Phi_n: B_n \rightarrow Ord,$$
for the purpose of induction.

To be more precise,
let $\varphi_\alpha :Ord\r Ord$ be the functions with
$$\varphi_1(\beta)=\omega^{1+\beta},$$
and $\varphi_\alpha(\beta)=$ the $\beta$'th fixed point of all $\varphi_{\alpha'}$'s where $\alpha' < \alpha +1$.
The $\varphi_\alpha$'s are the Veblen functions with the subscripts shifted by $1$, for $\alpha\geq 2$.
For every $\alpha\in Ord$, $\varphi_\alpha$ is continuous and strictly increasing.

We will now define functions
$$\Phi_n: B_n \rightarrow Ord$$
inductively.

For $x\in B_1$, $\alpha_1,\dots,\alpha_{m_x}\in Ord$ , define
$$\Phi_1(x;\alpha_1,\dots,\alpha_{m_x})=\alpha_1 +\dots +\alpha_{m_x}.$$

Now, suppose we have defined $\Phi_{n-1}(z;\alpha_1,\dots,\alpha_{m_z})$, for every $z\in B_{n-1},\; \alpha_1,\dots,\alpha_{m_z}\in Ord$.
Fix $z\in B_n,\; \alpha_1,\dots,\alpha_{m_z}\in Ord$. We can always write
$$z=((z_0,z_1,\dots,z_{\ell}),(i_1,\dots,i_{\ell}))=$$ 
$$=\gamma_{n-1,i_{\ell}}(\gamma_{n-1,i_{\ell-1}}(\dots \gamma_{n-1,i_2}(\gamma_{n-1,i_1}(z_0,z_1),z_2),\dots),z_{\ell})$$
with $i_1>\dots>i_\ell$, $z_0\in B_{n-1}$, and $z_1,\dots,z_{\ell}\in B_n$.
Let
$$\beta_{i,j}=\alpha_{1+m_{z_1}+\dots+m_{z_{i-1+j}}}.$$
Define inductively
$$\begin{array}{c}
\Phi_n(z;\alpha_1,\dots,\alpha_{m_z}):=\\
\Phi_{n-1}(z_0;\underbrace{1,\dots,1}_{m_{z_1}-1},\varphi_{n-1}(\Phi_n(z_1;\beta_{1,1},\dots,\beta_{1,m_{z_1}})),\underbrace{1,\dots,1}_{m_{z_2}-1},\\
\varphi_{n-1}(\Phi_n(z_2;\beta_{2,1},\dots,\beta_{2,m_{z_2}})), \underbrace{1,\dots,1}_{m_{z_3}-1}, \dots,\\
\underbrace{1,\dots,1}_{m_{z_{\ell}}-1},\varphi_{n-1}(\Phi_n(z_{\ell};\beta_{\ell,1},\dots,\beta_{\ell,m_{z_{\ell}}}))).\\
\end{array}$$
Note that not all of the arguments get used in the definition.
However, I  do not think that eliminating this ``wastefulness" would extend the
range of ordinals expressible by $B_n$.

\begin{theorem}\label{phisontoordinalwithalphain1orimage}
Fix an $n$. Then for every $\beta\in Ord$, there exists some $z\in B_n$ and some $\alpha_i\in \{1\}\cup Im(\varphi_n)$ such that
$$\beta=\Phi_n(z;\alpha_1,\dots,\alpha_{m_z}).$$
\end{theorem}

\begin{proof}
Induction.
Suppose n=1.
Then $$\Phi_1(n;\alpha_1,\dots,\alpha_n)=\alpha_1+\dots +\alpha_n.$$
Then for every $\beta\in Ord$, there exists an $n$ and $\alpha_1,\dots,\alpha_n\in\{1\}\cup\{\omega^\alpha|\alpha\in Ord\} =\{1\}\cup Im(\varphi_1)$ such that $\beta=\alpha_1+\dots+\alpha_n$.

Now, suppose for every $\beta\in Ord$ there is a $z\in B_{n-1}$ and $\alpha_1,\dots,\alpha_{m_z}\in \{1\}\cup Im(\varphi_{n-1})$ with
$$\beta =\Phi_{n-1}(z;\alpha_1,\dots,\alpha_{m_z}).$$
If $\alpha_1,\dots,\alpha_{m_z}=1$, we are done.
Suppose $\alpha_i=\varphi_{n-1}(\alpha_i')$. If $\alpha_i=\alpha_i'$, then $\alpha_i$ is a fixed point of $\varphi_{n-1}$. So there exists an $\alpha_i''$ with
$$\alpha_i=\varphi_n(\alpha_i'').$$
So we are done.
Suppose $\alpha_i'<\alpha_i$. Then we use the induction hypothesis on $\alpha_i'$.
\end{proof}

\vspace{5mm}

By Theorem \ref{phisontoordinalwithalphain1orimage}, the image of the function
$$\Phi_n :B_n\r Ord$$
given by $\Phi_n(x):=\Phi_n(x;1,\dots,1)$ is $\varphi_n(0)$, for $n=1,2,\dots$.

\begin{comment}
It is readily possible to extend the description of $B_n$ (and, in addition, $\mathscr{B}_n$) to $B_{\alpha}$, ( and $\mathscr{B}_\alpha$) for $\alpha\in Ord$,
by taking unions for limit ordinals $\alpha$. Then also our definition of $\Phi_n$ extends to $\Phi_\alpha$, $\alpha\in Ord$.
Therefore, the first ordinal $\gamma$ for which
$$Im(\Phi_{\gamma})=\gamma$$
is the Feferman-Sch\"{u}tte ordinal $\gamma=\Gamma_0$.
This can be compared with \cite{KockJoyalBatMascari}, where the authors consider the similar concept of opetopes and
possible machine implementations.
\end{comment}

\end{document}